\documentclass[11pt]{amsart}
\usepackage{amsxtra,amssymb,amsmath,amsthm,amsbsy}
\usepackage[all]{xy}

\usepackage{enumerate}

\numberwithin{equation}{section}

\theoremstyle{plain}
\newtheorem{theorem}{Theorem}[subsection]
\newtheorem{lemma}[theorem]{Lemma}
\newtheorem{proposition}[theorem]{Proposition}
\newtheorem{corollary}[theorem]{Corollary}

\theoremstyle{definition}
\newtheorem{definition}[theorem]{Definition}
\newtheorem{example}[theorem]{Example}
\newtheorem{remark}[theorem]{Remark}

\newcommand{\st} [1]     {\scriptscriptstyle{{#1}}}

\newcommand{\id}         {{\mathrm {Id}}}

\newcommand{\inv}        {\mathrm{inv}}

\def\transverse{\pitchfork}
\def\eps{\epsilon}

\def\V{\mathcal V}
\def\R{\mathbb R}
\def\C{C^\infty}

\def\id{{\rm id}}

\def\then{\Rightarrow}

\def\xto{\xrightarrow}
\def\xfrom{\xleftarrow}
\def\action{\curvearrowright}
\def\toto{\rightrightarrows}

\def\<{\langle}
\def\>{\rangle}

\newcommand{\SP} [1]     {{\left\langle {{#1}} \right\rangle}}

\newcommand{\Rm}         {(\mathbb{R},\cdot)}

\newcommand{\sour}        {\mathsf{s}}
\newcommand{\tar}         {{\mathsf{t}}}

\newcommand{\Lie}        {{\rm Lie}}

\newcommand{\dd}     {\mathrm{d}}

\begin{document}

\title[]
{Vector bundles over Lie groupoids and algebroids}

\author[]
{Henrique Bursztyn \and Alejandro Cabrera \and Matias del Hoyo}

\address
{IMPA, Estrada Dona Castorina 110, Rio de Janeiro, 22460-320, Brazil
} \email{henrique@impa.br, mdelhoyo@impa.br}

\address{Departamento de Matem\'atica Aplicada - IM, Universidade
Federal do Rio de Janeiro, CEP 21941-909, Rio de Janeiro, Brazil}
\email{acabrera@labma.ufrj.br}

\address{}
\email{}

\date{}


\begin{abstract}
We study VB-groupoids and VB-algebroids, which are vector bundles in
the realm of Lie groupoids and Lie algebroids. Through a suitable
reformulation of their definitions, we elucidate the Lie theory
relating these objects, i.e., their relation via differentiation and
integration. We also show how to extend our techniques to describe
the more general Lie theory underlying double Lie algebroids and
LA-groupoids.
\end{abstract}

\maketitle

\vspace{-1cm}

\tableofcontents

\vspace{-1cm}


\section{Introduction}\label{sec:intro}


Lie groupoids arise in various areas of geometry and topology, such
as group actions, foliations and Poisson geometry
\cite{CW,Mac-book,mm}, often serving as models for singular spaces
(see e.g. \cite{survey} and references therein). They provide a
unifying viewpoint to seemingly unrelated questions that has led to
important extensions of classical geo\-metrical results. Lie
algebroids are their infinitesimal counterparts, and both are
related by a rich Lie theory \cite{CF}, with many applications
beyond the classical theory of Lie algebras and Lie groups.


The main objects of study in this paper are the so-called {\em
VB-groupoids} and {\em VB-algebroids}
\cite{GM,mac-double1,dla,Prad2}, which can be thought of as
(categorified) vector bundles in the realm of Lie groupoids and Lie
algebroids. Paradigmatic examples include the tangent and cotangent
bundles of Lie groupoids and Lie algebroids. These objects have been
the subject of extensive study in the last years
\cite{mac-double1,mac-double2,mac-crelle}, partly motivated by their
deep ties with Poisson geometry \cite{mk2,Mac-Xu,Mac-Xu2}.


From another perspective, VB-groupoids and VB-algebroids are
intimately related to the study of representations. Indeed,
representations of Lie groupoids on vector bundles provide a wealth
of examples of VB-groupoids through the construction of semi-direct
products, and analogously for Lie algebroids. More generally, it is
shown in \cite{GM,GM2} that VB-groupoids and VB-algebroids provide
an intrinsic approach to representations up to homotopy
\cite{AC1,AC2}, a ``higher'' notion of representation needed to make
sense e.g. of the adjoint representation of a Lie groupoid or
algebroid.


Our main goal in this paper is to describe the Lie theory relating
VB-algebroids and VB-groupoids, i.e., to elucidate how they are
related via differentiation and integration. A key step in our work
relies on finding simpler formulations of VB-groupoids and
VB-algebroids, explained in Theorems \ref{thm:VB-gpd} and
\ref{thm:VB-alg}. We show that VB-groupoids (resp. VB-algebroids)
can be described as Lie groupoids (resp. Lie algebroids) equipped
with an additional action of the monoid $\Rm$, with natural
compatibility conditions, in the spirit of the characterization of
vector bundles in \cite{GR}. From this viewpoint, we can study the
Lie functor relating VB-groupoids and VB-algebroids by
differentiating and integrating these actions. We prove in
Theorem~\ref{thm:integration} that, if the total algebroid of a
VB-algebroid is integrable, then its vector-bundle structure can be
lifted to the source-simply-connected Lie groupoid integrating it,
which then becomes a VB-groupoid. This result finds applications
e.g. in the study of Dirac structures on Lie groupoids \cite{ortiz}
and also provides information about integration of representations
up to homotopy \cite{AS2} (see also \cite{BrCO}).


Our techniques to handle VB-algebroids and VB-groupoids allow us to
go further and explain the Lie theory relating more general objects,
known as double Lie algebroids and LA-groupoids
\cite{mac-double1,mac-double2,Mac04,mac-crelle}. One can think of
them as Lie algebroids defined over Lie algebroids and Lie
groupoids, or as generalizations of VB-algebroids and VB-groupoids
in which the vector-bundle structures are enhanced to be Lie
algebroids. In this more general context, we prove in
Theorem~\ref{thm:LADLA} that if the top Lie algebroid in a double
Lie algebroid is integrable, then its source-simply-connected
integration naturally becomes an LA-groupoid; this provides the
reverse procedure to the differentiation in \cite{mac-double2}. Our
approach to establish this result heavily relies on the well-known
duality between Lie algebroids and linear Poisson structures. Rather
than treating double Lie algebroids and LA-groupoids directly, we
focus on their dual objects. Building on our previous results for
VB-groupoids and VB-algebroids, we describe these duals as Lie
bialgebroids and Poisson groupoids carrying an extra compatible
$\Rm$-action. The result then follows from the differentiation and
integration properties of these actions, along with a natural
integration result for morphisms of Lie bialgebroids (see
Prop.~\ref{bialgmap}).

Throughout the paper, several arguments rely on the construction of
fibred products in the categories of Lie algebroids and Lie
groupoids; we collect the necessary results in the appendix,
organizing and extending previous discussions about fibred products
in the literature.

\medskip

\noindent{\bf Organization.}
After preliminaries in
Section \ref{sec:prelim}, which include the description of (double)
vector bundles and linear Poisson structures in terms of
$\Rm$-actions, we present new characterizations of VB-groupoids and
VB-algebroids in Section \ref{sec:regactions}. Their Lie theory is
explained in Section \ref{sec:Lie}. In Section \ref{sec:DLA}, we
consider double Lie algebroids and LA-groupoids, describe their dual
objects, and extend the Lie theory of Section \ref{sec:Lie} to this
context.

\medskip

\noindent{\bf Acknowledgments.} We are grateful to several
institutes for hosting us during different stages of this project,
especially IST (Lisbon) and U. Utrecht. We thank several people for
useful discussions and helpful advice, including M. Crainic, T.
Drummond, R. Fernandes, M. Jotz, D. Li-Bland, K. Mackenzie, R.
Mehta, E. Meinreken, and C. Ortiz. We thank CNPq and FAPERJ for
financial support.


\medskip

\noindent{\bf Notation and conventions.}
A Lie groupoid with arrows manifold $G$ and units $M$ will be
denoted by $ G\toto M,$ or simply by $G$. We write $\sour_G$ and
$\tar_G$ for the source and target maps. The set
$G^{(2)}=\{(g,h):\sour_G(g)=\tar_G(h)\}\subset G\times G$
is the domain of the multiplication  $m_G:G^{(2)}\to G$,
$m(g,h)=gh$. The unit map $u_G: M \to G$ is often used to identify
$M$ with its image in $G$, and we write $\inv_G:G\to G$, $\inv_G(g)
= g^{-1}$, for the inversion. We often suppress the subscript $G$
and simply write $\sour,\tar, m, u, \inv$. We denote Lie-groupoid
maps by $(\Phi,\phi): (G_1\toto M_1)\to (G_2\toto M_2)$, or simply
by $\Phi: G_1\to G_2$, having in mind that $\phi = \Phi|_M$.

Given a Lie groupoid $G\toto M$, its Lie algebroid
$\mathrm{Lie}(G)=A_G$ has underlying vector bundle
$A_G=\ker(\dd\sour)|_M \to M$, with anchor map
$\dd\tar|_{A_G}:A_G\to M$ and Lie bracket on sections of $A_G$
induced by right-invariant vector fields. A general Lie algebroid
with underlying vector bundle $A\to M$ will be usually denoted by $
A \then M.$ We use the notation $\rho_{\st{A}}$ for the anchor map
and $[\cdot,\cdot]_{\st{A}}$ for the bracket, or simply $\rho$ and
$[\cdot,\cdot]$, if there is no risk of confusion.


\section{Preliminaries}\label{sec:prelim}


We start by discussing a characterization of vector bundles via
actions of the multiplicative monoid $\Rm$ as in \cite{GR}, where
details can be found. We will also consider double vector bundles
and linear Poisson structures from this viewpoint.

\subsection{A characterization of vector bundles}


Let $D$ be a smooth manifold, and denote by $\Rm$ the multiplicative
monoid of real numbers. An {\em action} $h:\Rm\action D$ of $\Rm$ on
$D$ is a smooth map
$$
h:\R\times D \to D, \qquad h(\lambda,x)=h_\lambda(x),
$$
satisfying the usual action axioms: $h_1 = \id_D$ and $h_\lambda
h_{\lambda'}=h_{\lambda\lambda'}$ for all $\lambda,\lambda'\in\R$.


Assume that $D$ is connected. Since the map $h_0$ is a projection,
i.e. $h_0\circ h_0=h_0$, it follows that $h_0(D)\subset D$ is an
embedded submanifold with $T_{h_0(x)} h_0(D)=\dd_xh_0(T_xD)$ for all
$x\in D$, see e.g. \cite[Thm 1.13]{nodg}. Using that $h$ is an
action one may check that the rank of the map $h_0: D\to h_0(D)$ is
constant, and hence it is a surjective submersion. When $D$ is not
connected, the rank of $h_0$ is only locally constant, i.e., it is
constant on each connected component, but may vary from one
component to another.
When considering $\Rm$-actions on disconnected manifolds, {\em we
will always assume that $h_0$ has constant rank}. This guarantees
that $h_0(D)$ is an embedded submanifold of $D$.


The key example of an action of $\Rm$ is the fibrewise scalar
multiplication (homotheties) on a vector bundle $E\to M$, in which
case $h_0(E)=M$. This action satisfies an additional property: if
$x\in E$ is a non-zero vector then the curve $\lambda\mapsto
h_\lambda(x)$ has non-zero velocity at the origin. This motivates
the following definition.


\begin{definition}
We call an action $h: \Rm \action D$ {\em regular} if the following
equation holds at all points in $x\in D$:
\begin{equation}\label{eq:reg}
\frac{d}{d\lambda}\Big|_{\lambda=0} h_\lambda(x)=0 \quad\then\quad
x=h_0(x).
\end{equation}
\end{definition}

It turns out that an action is regular if and only if it can be
realized as the homotheties of a vector bundle. Let us recall a
construction from \cite{GR} that explains this fact and plays a key
role in this paper.


Given an action $h: \Rm \action D$, there is always a vector bundle
over $h_0(D)$ canonically associated with it, the so-called {\em
vertical bundle}, defined by
\begin{equation}\label{eq:VE}
V_h D=\ker(\dd h_0)|_{h_0(D)}.
\end{equation}
Note that its underlying $\Rm$-action is the restriction of the
homotheties on $TD\to D$. This passage from $\Rm$-actions to vector
bundles is functorial, i.e., an $\Rm$-equivariant map $D_1\to D_2$
yields a canonical vector bundle map $V_hD_1\to V_hD_2$, and this
assignment respects identities and compositions.


The {\em vertical lift} $\V_h: D\to V_hD$ is the smooth map that
associates to each point $x\in D$ the velocity at time 0 of the
curve $\lambda\mapsto h_\lambda(x)$:
\begin{equation}\label{eq:vlift}
\V_h(x)=\frac{d}{d\lambda}\Big|_{\lambda=0}h_\lambda(x),
\end{equation}
so the action is regular if and only if the zeroes of $\V_h$ are
exactly its fixed points. One may readily verify (through the chain
rule) that the vertical lift is $\Rm$-equivariant. When $h$ is
defined by homotheties of a vector bundle, the vertical lift is the
standard identification with its associated vertical bundle.


The regularity of an action is clearly a necessary condition for the
vertical lift to be a diffeomorphism. The less evident fact is that
it is also sufficient:
\begin{theorem}[\cite{GR}]
\label{thm:vertical.lift} An action $h:\Rm\action D$ is regular if
and only if the vertical lift $\V_h$ is a diffeomorphism onto
$V_hD$. In this case $D\to h_0(D)$ inherits a natural vector-bundle
structure for which $\V_h$ is a vector-bundle isomorphism.
\end{theorem}

This theorem sets an equivalence between regular $\Rm$-actions and
vector bundles, and allows the theory of vector bundles to be
rephrased in terms of $\Rm$-actions. For instance, as immediate
consequences of Theorem~\ref{thm:vertical.lift}, we see that a
vector subbundle is the same as an invariant submanifold, and that a
vector-bundle map is the same as a smooth equivariant map. (Note
that, by continuity, it is enough to check equivariance for
$\lambda\neq 0$.) For more details, see \cite{GR}.


\begin{remark}\label{rem:regular}
Denote by $\mathrm{VB}$ the category of vector bundles and by
$\mathrm{ACT}$ that of $\Rm$-actions. By considering the vertical
bundles associated with actions and the actions by homotheties
underlying vector bundles, we obtain a pair of functors
$$
\xymatrix@1{\mathrm{ACT} \ar@<0.2pc>[r]^V & \ar@<0.2pc>[l]^U \mathrm{VB} }.
$$
The vertical lift $\V:\id_{\mathrm{Act}} \to U\circ V$ is a natural
transformation that is invertible over the image of $U$. It easily
follows that $V$ is a left adjoint for $U$, that $U$ is fully
faithful, and hence $\mathrm{VB}$ is a {\it co-reflective
subcategory} of $\mathrm{ACT}$.
From this perspective, $V$ is a projection that associates to any
action a regular one, so we may think of it as a ``regularization
functor''.
\end{remark}


\begin{remark}\label{rem:factorization}
If $D$ is a manifold equipped with an action $h:\Rm\action D$, regular or not,
the vertical lift map $\V_h:D\to TD$ can be expressed
as the following composition:
\begin{equation}\label{eq:jdiag}
\begin{matrix}
\xymatrix{ D \ar[rr]^{\V_h} \ar[dr]_{l} & & TD \\ & TD\times T\R
\ar[ur]_{\dd h} }
\end{matrix}
\end{equation}
where $l(x)=((x,0),(0,\partial_\lambda))$ is the map whose first
component is the zero section of $TD$ and whose second component is
a constant map. The factorization $\V_h= \dd h \circ l$ will be
useful in subsequent sections.
\end{remark}


\subsection{Double vector bundles}


A {\em double vector bundle} $(D,E,A,M)$ is a commutative diagram
\begin{equation}\label{eq:D}
\begin{matrix}
D & \to & E \\ \downarrow & & \downarrow\\ A & \to & M
\end{matrix}
\end{equation}
in which every arrow is a vector bundle and so that the two
vector-bundle structures on $D$ are {\em compatible}, in the sense
that the structural maps of one (projection, zero section, fibrewise
addition and multiplication by scalars) are vector-bundle maps with
respect to the other (see \cite[Prop.~2.1]{GM}). Whenever we need to
specify the structure maps involved in a double vector bundle, we
use the notation $q^D_E$ for the bundle projection $D\to E$, $0^D_E$
for the corresponding zero section, and similarly for the structure
maps of the other vector bundles.


For double vector bundles $(D, E, A, M)$ and
$(\tilde{D},\tilde{E},\tilde{A},\tilde{M})$, a map $\Psi: D\to
\tilde{D}$ is a {\em morphism} if it gives rise to vector-bundle
maps $(\Psi,\psi_A):(D\to A)\to (\tilde{D}\to \tilde{A})$ and
$(\Psi,\psi_E): (D\to E)\to (\tilde{D}\to \tilde{E})$. It follows
that $\psi_A: A\to \tilde{A}$ and $\psi_E: E\to \tilde{E}$ are also
vector-bundle maps, covering the same map $\psi_M: M\to \tilde{M}$.
Identifying $M,A,E$ with submanifolds of $D$ via the corresponding
zero sections, the maps $\psi_E,\psi_A,\psi_M$ are just the
restrictions $\Psi|_E,\Psi|_A,\Psi|_M$.


For a double vector bundle as in \eqref{eq:D} the bundles $E\to M$
and $A \to M$ are called the {\it side bundles}. The intersection of
the kernels of the projections $q^D_A: (D\to E)\to (A\to M)$ and
$q^D_E: (D\to A)\to (E\to M)$ defines another vector bundle $C\to
M$, known as the {\em core} bundle of $D$. The core and side bundles
are central ingredients in the structure of $D$: there always exists
a (non-canonical) splitting
$$
D\xto\sim A\oplus C \oplus E,
$$
i.e., an isomorphism inducing the identity on the sides and core,
where the triple sum is regarded as a double vector bundle in the
obvious way (see e.g. \cite{GM}).


\begin{example}\label{ex:tangcot}
The main examples of double vector bundles are the tangent and
cotangent bundles of a vector bundle $A\to M$ (see e.g. \cite[\S
9.4]{Mac-book}):
\begin{equation}\label{eq:tangcot}
\begin{matrix}
TA & \to & TM \\ \downarrow & & \downarrow\\ A & \to & M,
\end{matrix}
\qquad\qquad
\begin{matrix}
T^*A & \to & A^* \\ \downarrow & & \downarrow\\ A & \to & M.
\end{matrix}
\end{equation}
If $h$ denotes the $\Rm$-action on $A$ by homotheties, the action
corresponding to the bundle structure $TA\to TM$ is $\lambda\mapsto
\dd h_\lambda$. The action corresponding to $T^*A\to  A^*$,
sometimes referred to as the {\em phase lift}, will be described in
terms of $h$ after Prop.~\ref{tildeh} below. Note that the core of
$TA$ is the vertical bundle $VA\to M$, which is isomorphic to $A\to
M$. The core of $T^*A$ can be identified with $T^*M\to M$.
\end{example}


\begin{remark}[The reversal isomorphism]
For a vector bundle $A\to M$, there is a canonical isomorphism of
double vector bundles,
\begin{equation}\label{eq:R}
R_A: T^*A\to T^*A^*,
\end{equation}
known as the {\em reversal isomorphism}, preserving side bundles and
restricting to $-\id$ on the cores. In local coordinates we have
splittings $T^*A\cong A\oplus T^*M\oplus A^*$ and $T^*A^*\cong A^*\oplus T^*M\oplus A$, with respect to which
$R_A(\phi,\omega,v)=(v,-\omega,\phi)$, and this turns out to be well
defined globally. For a detailed discussion, see e.g. \cite[\S
9.5]{Mac-book}.
\end{remark}



A particularly rich aspect of the theory of double vector bundles
concerns the notion of duality, that we need to recall from
\cite[\S~9.2]{Mac-book}. Associated to a double vector bundle
$(D,E,A,M)$ we have a {\em horizontal dual} and a {\em vertical
dual},
$$
\begin{matrix}
D^*_E & \to & E \\ \downarrow & & \downarrow\\ C^* & \to & M,
\end{matrix}
\qquad\qquad
\begin{matrix}
D^*_A & \to & C^* \\ \downarrow & & \downarrow\\ A & \to & M,
\end{matrix}
$$
which are double vector bundles containing the dual of the core
bundle of $D$ as side bundles, and whose cores are $A^*\to M$ and
$E^*\to M$, respectively. For example, given a vector bundle $A\to
M$, the vertical dual of its tangent bundle is its cotangent bundle,
as depicted in \eqref{eq:tangcot}, while the horizontal dual is a
new double vector bundle $((TA)^*_{TM},TM,A^*,M)$.


It is often convenient to think of a double vector bundle
\eqref{eq:D} and its two duals as parts of a larger object, the
so-called {\em cotangent cube}
\begin{equation}\label{cot.cub}
\begin{matrix}
\xymatrix@R=6pt@C=6pt{
T^*D \ar[rr] \ar[dd] \ar[rd]&  & {D^*_E} \ar'[d][dd] \ar[dr]& \\
 & D \ar[dd] \ar[rr]&  & E\ar[dd]\\
D^*_A \ar'[r][rr] \ar[dr]&  & C^* \ar[dr]& \\
& A \ar[rr]& & M.}
\end{matrix}
\end{equation}
The horizontal and vertical duals of a double vector bundle are
related by a natural pairing $D_A^*\times_{C^*} D_E^*\to
\mathbb{R}$, which induces an isomorphism of double vector bundles
\begin{equation}\label{eq:Z} Z_D: D_A^*\to (D_E^*)^*_{C^*},
\end{equation}
interchanging the side bundles and inducing $-\id$ on the cores (cf.
\cite[Thms.~9.2.2 \& 9.2.4]{Mac-book}). This shows that, when
considering the double vector bundles $D$, $D_A^*$ and $D_E^*$,
taking further (horizontal or vertical) duals basically interchanges
them.


\begin{remark}\label{rem:compatible}
For later use, we recall the following compatibility between the
isomorphisms \eqref{eq:R} and \eqref{eq:Z}. Given $D$ as in
\eqref{eq:D}, consider the induced vector-bundles $T^*D\to D_A^*$
and $T^*D_E^*\to (D_E^*)^*_{C^*}$. Then the reversal isomorphism
associated with $D\to E$ preserves these bundle structures and
covers \eqref{eq:Z}; i.e., the following square commutes (see
\cite[Thm.~6.1]{Mac04}):
$$
\xymatrix{T^*D  \ar[r]^R \ar[d] & T^*D^*_E \ar[d]
\\
D^*_A \ar[r]_Z & (D^*_E)^*_{C^*},}
$$
The pair $(R,Z)$ actually defines an isomorphism between the
cotangent cubes of $D$ and $D^*_E$, that we may see as a higher
analogue of the reversal isomorphism \eqref{eq:R}.
\end{remark}


Double vector bundles admit a simple characterization in terms of
regular actions. In a double vector bundle $(D, E, A, M)$, the
actions $h,k:\Rm\action D$ corresponding to $D\to A$ and $D\to E$
commute, i.e., $h_\lambda k_\mu = k_\mu h_\lambda$ for all
$\lambda,\mu\in\R$. Conversely, if a manifold $D$ is endowed with
two commuting regular actions $h,k:\Rm\action D$, then in light of
Theorem~\ref{thm:vertical.lift} we get a commutative diagram of
vector bundles
$$
\begin{matrix}
D & \to & k_0(D) \\ \downarrow & & \downarrow\\ h_0(D) & \to & h_0k_0(D).
\end{matrix}
$$
To see that this is a double vector bundle, note that, since $h$ and
$k$ commute, the vertical lift $\V_h:D\to TD$ intertwines $k$ and
$\dd k$, so we can embed the previous square into the double vector
bundle $(TD,\dd k_0(TD),D,k_0(D))$, from where it inherits the
required compatibility condition. Hence we conclude (see \cite[Thm
3.1]{GR}):

\begin{proposition}\label{prop:DVB}
There is a one-to-one correspondence between double vector bundle
structures on $D$ and pairs of commuting regular actions $\Rm\action
D$.
\end{proposition}


For double vector bundles $D$ and $\tilde{D}$, defined by
$\Rm$-actions $h,k$ and $\tilde{h}, \tilde{k}$ respectively, a map
$D\to \tilde{D}$ is a morphism if and only if it is equivariant for
both actions, i.e.,  it intertwines $h$ and $\tilde{h}$ as well as
$k$ and $\tilde{k}$.


Let us now discuss the behavior of the $\Rm$-actions under duality
of double vector bundles. Given $D$ as in \eqref{eq:D} and its
vertical dual, let $h$, $k$, $\bar{h}$ and $\bar{k}$ denote the
actions corresponding to $D\to A$, $D\to E$, $D^*_A\to A$ and
$D^*_A\to C^*$, respectively. Then the restrictions $k|_A$ and
$\bar{k}|_A$ agree, and $k$ and $\bar{k}$ are related by the
following equation (see \cite[pp. 349]{Mac-book}):
\begin{equation}\label{pairing}
\<\bar{k}_\lambda (\xi),k_\lambda(v)\>=\lambda\<\xi,v\>,\qquad a\in
A,\ \xi\in (D^{*}_A)_a,\ v\in D_a.
\end{equation}

We will relate the homotheties of the dual, $\bar{k}_\lambda$, with
the dual relation of the homotheties $k_{{\lambda}^{-1}}$. Recall
that if $\phi:(E\to M)\to(\tilde{E}\to \tilde{M})$ is a map of
vector bundles, then its {\em dual relation} is the relation defined
by
\begin{equation}\label{eq:dual}
\phi^* := \{\big(\phi^*(\xi),\xi\big),\; x\in M,\; \xi\in
\tilde{E}^*_{\phi(x)}\} \subset E^*\times \tilde{E}^*,
\end{equation}
and if the map $\phi$ is invertible, then \eqref{eq:dual} is the
graph of an actual vector bundle map $\tilde{E}^*\to E^*$, still
denoted by $\phi^*$, that agrees with $\phi^{-1}$ on the base.

\begin{proposition}\label{tildeh}
For $\lambda\neq 0$, the following is an identity of vector-bundle
maps:
$$
\bar{k}_\lambda = (k_{{\lambda}^{-1}})^*\bar{h}_\lambda =
(k^{-1}_\lambda)^*\bar{h}_\lambda:(D^*_A\to A)\to (D^*_A\to A).
$$
\end{proposition}
\begin{proof}
Let us show that the maps agree over the base and on each fiber. If
$a\in A$, since $k|_A= \bar{k}|_A$ and $\bar{h}|_A$ is trivial, we
have $\bar{k}_\lambda (a) = k_\lambda (a) = (k^{-1}_\lambda)^* (a) =
(k^{-1}_\lambda)^* \bar{h}_\lambda (a)$. Now let $\xi\in (D_A^*)_a$,
so $\xi:D_a\to\R$ is a linear map. Since the linear structure on
$D_a$ is given by $h_\lambda$ we have $\<\bar{h}_\lambda(\xi),v\> =
\lambda \<\xi,v\>$, for $v\in D_a$. Therefore
$$
\<(k_\lambda^{-1})^*\bar{h}_\lambda (\xi), k_\lambda (v)\> =
\<\bar{h}_\lambda (\xi) , v\> = \lambda \<\xi,v\>,
$$
which shows that $(k_\lambda^{-1})^*\bar{h}_\lambda (\xi)$ must be
$\bar{k}_\lambda (\xi)$, by \eqref{pairing}.
\end{proof}


As a corollary we obtain an explicit description of the phase lift
(see Example~\ref{ex:tangcot}): if $h:\Rm \action E$ is a regular
action, then the action $\Rm \action T^*E$ corresponding to $T^*E\to
E^*$ is given, for each $\lambda\neq 0$, by
\begin{equation}\label{eq:phase}
\lambda \cdot (\dd h_{{\lambda}^{-1}})^*:T^*E\to T^*E,
\end{equation}
where $\lambda \cdot(-)$ stands for the multiplication by $\lambda$
in the canonical structure $T^*E\to E$.


\subsection{Linear Poisson structures}
\label{subsec:linear}


Given a vector bundle $q: E\to M$, a function $f\in\C(E)$ is said to
be {\em linear} (resp. \emph{basic}) if it is linear (resp.
constant) when restricted to each fiber. Linear functions are in
one-to-one correspondence with sections of the dual bundle,
\begin{equation}\label{eq:linear}
\Gamma(E^*)\ni \xi \mapsto \ell_\xi \in C^\infty(E),\;\;
\ell_\xi(v)=\SP{\xi,v},
\end{equation}
while basic functions correspond to pullbacks of functions defined over the base,
$$
C^\infty(M) \ni f \mapsto q^*f \in C^\infty(E).
$$


A {\em linear Poisson structure} on $E\to M$ is a Poisson structure
$\{\cdot,\cdot\}$ on the total space $E$ of the vector bundle $E\to
M$ which satisfies the following:
\begin{enumerate}
\item[$(i)$] $f,g$ linear $\then$ $\{f,g\}$ linear,
\item[$(ii)$] $f$ linear, $g$ basic $\then$ $\{f,g\}$ basic,
\item[$(iii)$] $f,g$ basic $\then$ $\{f,g\}=0$.
\end{enumerate}


Linear Poisson structures can be also described by means of Poisson
bivector fields $\pi \in \Gamma(\wedge^2TE)$: a Poisson structure
$\pi$ on $E$ is linear if and only if $\pi^\#:T^*E\to TE$ yields a
map of double vector bundles (see e.g. \cite[Sec.~7.2]{KU}):
\begin{equation}\label{eq:lps}
 \begin{matrix}
T^*E & \to & E^* \\ \downarrow & & \downarrow\\ E & \to & M
\end{matrix}
\qquad\xto{\pi^\#}\qquad
\begin{matrix}
TE & \to & TM \\ \downarrow & & \downarrow\\ E & \to & M.
\end{matrix}
\end{equation}
An example is given by the canonical Poisson structure on $E=T^*M$.


The following is an alternative characterization of linear Poisson
structures via $\Rm$-actions:

\begin{proposition}
\label{linear} Let $E\to M$ be a vector bundle, with regular action
$h:\Rm\action E$, and let $\pi$ be a Poisson bivector field on $E$.
Then the Poisson structure is linear if and only if
$h_\lambda:(E,\pi)\to(E,\lambda\pi)$ is a Poisson map for all
$\lambda \neq 0$.
\end{proposition}
\begin{proof}
The linearity of the Poisson structure is equivalent to the
condition that the map $\pi^\#: T^*E\to TE$ intertwines the actions
corresponding to the bundles $T^*E\to E^*$ and $TE\to TM$. This
means that, for each $\lambda\neq 0$,
$$
\pi^\# (\lambda (\dd h_{{\lambda}^{-1}})^*) = (\dd h_{\lambda})
\pi^\#,
$$
(see Example~\ref{ex:tangcot} and \eqref{eq:phase}), which can be
re-written as $\lambda \pi^\# = (\dd h_{\lambda}) \pi^\# (\dd
h_{\lambda})^*$, exactly the condition for
$h_\lambda:(E,\pi)\to(E,\lambda\pi)$ being a Poisson map.
\end{proof}

One can immediately apply the previous proposition to characterize
{\em double linear Poisson structures}, i.e., Poisson structures on
double vector bundles which are linear with respect to both
vector-bundle structures (see \cite[Sec.~3.4]{GM}).


A key fact, to be used recurrently throughout the paper, is the
duality between linear Poisson structures and Lie algebroids. For a
vector bundle $A\to M$, there is a one-to-one correspondence between
Lie algebroid structures $A\then M$ and linear Poisson structures on
the dual $A^*\to M$ (see e.g. \cite[Chp.~10]{Mac-book}), via the
relations
$$
\{\ell_X,\ell_Y\}=\ell_{[X,Y]}, \;\; \{\ell_X,q^*f\}=q^*(\rho(X)f),
$$
for $X,Y\in \Gamma(A)$, $f\in C^\infty(M)$. As for the behavior of
maps under this correspondence, if $\phi:(A_1\to M_1)\to (A_2 \to
M_2)$ is a map between the underlying vector bundles of given Lie
algebroids, then it is a Lie algebroid map if and only if its dual
relation $\phi^*$ (see \eqref{eq:dual}) is coisotropic in
$A_1^*\times \overline{A_2^*}$ (here $^-$ indicates that $A_2^*$ is
equipped with the opposite Poisson structure).
We will discuss more general instances of this duality in
Section~\ref{sec:DLA}.


\begin{remark}\label{rem:poisson}
Recall that a Poisson manifold has an induced Lie-algebroid
structure on its cotangent bundle, and hence a linear Poisson
structure on its tangent bundle, known as the {\em tangent lift}.
For a vector bundle $E$ equipped with a linear Poisson structure,
there are induced Lie-algebroid structures on $E^*$ and $T^*E$, and
these are compatible in the sense that the natural projection
$T^*E\to E^*$ is a Lie-algebroid map (cf.
\cite[Prop.~10.3.6]{Mac-book}). Moreover, the tangent-lift Poisson
structure on $TE$ is double linear (cf.
\cite[Thm.~10.3.14]{Mac-book}). 
\end{remark}


\section{VB-groupoids and VB-algebroids as regular actions}\label{sec:regactions}


Just as Lie groupoids generalize both smooth manifolds and Lie
groups, VB-groupoids simultaneously encompass  vector bundles and
representations of Lie groups, see e.g. \cite[\S 11.2]{Mac-book} and
\cite{GM2}. VB-algebroids are the analogous infinitesimal objects,
see e.g. \cite{GM,dla}. In this section we present new
characterizations of VB-groupoids and VB-algebroids in terms of
$\Rm$-actions. In spite of the clear analogy between the results, we
point out that the arguments justifying them will often differ in
spirit, reflecting the structural differences in how Lie algebroids
and groupoids are defined.

\subsection{VB-groupoids}\label{subsec:VBgrp}


A {\em VB-groupoid} (cf. \cite[Def.~3.3]{GM2}) consists of Lie
groupoids $\Gamma\toto E$, $G\toto M$, and vector bundles $\Gamma\to
G$, $E\to M$, forming a diagram
\begin{equation}\label{eq:VBgrd}
\begin{matrix}
\Gamma & \toto & E\\
\downarrow & & \downarrow\\
G& \toto & M
\end{matrix}
\end{equation}
that is compatible in the following sense: the groupoid structural
maps of $\Gamma$ (source, target, multiplication, unit, inverse)
cover the corresponding ones of $G$ and are vector-bundle maps. (To
consider the compatibility for the multiplication we view
$\Gamma\times_E\Gamma\to G\times_M G$ as a vector subbundle of the
product.) We will refer to $\Gamma\toto E$ as the {\em total
groupoid} and to $G\toto M$ as the {\em base groupoid}.


Given a VB-groupoid as above, its (right) {\em core} $C\to M$ is
defined as the kernel of the vector-bundle map $(u_G^*\Gamma\to M)
\stackrel{\sour_\Gamma}{\to} (E\to M)$. The core plays a key role in
the structure of a VB-groupoid: there is a short exact sequence of
bundles over $G$,
$$
0 \to \tar_G^*C\to \Gamma \xto{\sour_\Gamma} \sour_G^*E \to 0,
$$
called the {\em (right) core exact sequence}, and any splitting of
this sequence induces a decomposition of the VB-groupoid into the
base groupoid, the vector bundles $C$ and $E$, and some extra
algebraic data \cite{GM2} (see Remark~\ref{rmk:ruth}).


A {\em VB-groupoid map} (or {\em morphism})
$$
\begin{matrix}
\Gamma & \toto & E\\
\downarrow & & \downarrow\\
G& \toto & M
\end{matrix}
\qquad\to{}\qquad
\begin{matrix}
\tilde{\Gamma} & \toto & \tilde{E}\\
\downarrow & & \downarrow\\
\tilde{G} & \toto & \tilde{M}
\end{matrix}
$$
is defined by a Lie groupoid map $\Phi$ between the total Lie
groupoids which is linear. It follows that the restriction $\Phi|_E:
E\to \tilde{E} $ is also linear, $\Phi|_G: G\to \tilde{G}$ is a Lie
groupoid map, and core bundles are preserved.


We denote by $\mathrm{VB}(G\toto M)$ the category of VB-groupoids
over $G\toto M$, with morphisms being VB-groupoid maps restricting
to the identity on $G\toto M$.


\begin{remark}\label{rem:sscandmap}\
\begin{enumerate}[(a)]
\item  The core exact sequence implies that each
source-fiber of $\Gamma$ is an affine bundle over a source-fiber of
$G$, and hence $\Gamma$ is source-connected, or
source-simply-connected, if and only if so is $G$.

\item The core exact sequence is
natural with respect to maps, from where one verifies that a
VB-groupoid map $\Phi$ as above is fiberwise injective (resp.
surjective) if and only if it is so when restricted to the side
bundle $E$ and the core $C$.
\end{enumerate}
\end{remark}


\begin{example}[Tangent groupoid]
\label{ex:tangG}
Given a Lie groupoid $G\toto M$, its {\em tangent
groupoid} $TG\toto TM$ is defined by differentiating the structural
maps of $G\toto M$. It is naturally a VB-groupoid over $G\toto M$,
whose core $A_G\to M$ is the vector bundle of the Lie algebroid of
$G$.
One can readily check that the passage from Lie groupoids to their tangent groupoids is functorial.
\end{example}

\begin{example}
[Cotangent groupoid] Given a Lie groupoid $G\toto M$, it is also
possible to define its {\em cotangent groupoid}. Its space of arrows
consists of the cotangent bundle $T^*G$, and its objects are given
by the dual of the core, $A_G^*$. The structural maps of $T^*G\toto
A_G^*$ are described, e.g., in \cite[\S~11.3]{Mac-book}. This
groupoid structure makes $T^*G\toto A_G^*$ into a VB-groupoid over
$G\toto M$, with core $T^*M\to M$.
\end{example}

\begin{example}[Representations]
\label{ex:rep.groupoids} Given a representation of a Lie groupoid
$G\toto M$ on a vector bundle $E\to M$, the corresponding action
groupoid $G\times_M E\toto E$ is naturally a VB-groupoid over
$G\toto M$, whose core is trivial. One may directly verify that
every VB-groupoid with trivial core arises in this way.
\end{example}


There is a duality construction for VB-groupoids, for which the
cotangent groupoid $T^*G\toto A_G^*$ is the dual of $TG\toto TM$.
Given a VB-groupoid $\Gamma$ as in \eqref{eq:VBgrd}, its {\em dual
VB-groupoid} has the same base as \eqref{eq:VBgrd} and total
groupoid with arrows $\Gamma^*$ and objects $C^*$, the total space
of the dual of the core bundle $C\to M$,
\begin{equation}\label{eq:dualgp}
\begin{matrix}
\Gamma^* & \toto & C^*\\
\downarrow & & \downarrow\\
G& \toto & M.
\end{matrix}
\end{equation}

For the definition of the groupoid structural maps on $\Gamma^*\toto
C^*$, we refer to \cite[\S 11.2]{Mac-book}. Given a VB-groupoid map
$\Phi: \Gamma\to \tilde{\Gamma}$ covering the identity $G\to G$, the
dual map $\Phi^*:\tilde{\Gamma}^*\to \Gamma$ is a VB-groupoid map,
so taking duals defines an involutive functor on $\mathrm{VB}(G\toto
M)$.


\begin{remark}
\label{rmk:ruth} As suggested by Example \ref{ex:rep.groupoids}, one
can think of VB-groupoids as generalized representations \cite{GM2}.
By choosing a splitting of the core exact sequence of
\eqref{eq:VBgrd}, one can associate to every arrow in $G\toto M$ a
linear map between the fibers of the complex \smash{$C\xto{t_\Gamma}
E$}, and this yields a {\em representation up to homotopy}
\cite{AC2}. It is proven in \cite{GM2} that there is a one-to-one
correspondence between isomorphism classes of VB-groupoids and of
2-term representations up to homotopy.
\end{remark}


\subsection{Characterization as multiplicative actions}
\label{sec:multact}


We will now use Thm.~\ref{thm:vertical.lift} to establish a
characterization of VB-groupoids by means of $\Rm$-actions,
simplifying their original definition.

\begin{definition}\label{def:mult}
An $\Rm$-action $h$ on the space of arrows of a Lie groupoid
$\Gamma\toto E$ is called {\em multiplicative} if $h_\lambda:
\Gamma\to \Gamma$ is a groupoid map for each $\lambda\in
\mathbb{R}$.
\end{definition}

Note that the restriction $h|_E$ is an $\Rm$-action on $E$, which is
regular if $h$ is. It is often convenient to think of $h$ as a pair
of actions on $\Gamma$ and $E$, in which case we use the notation
$(h^\Gamma,h^E)$.


Recall that the fixed points of an $\Rm$-action $h$, i.e., the image
of $h_0$, define an embedded submanifold. For a multiplicative
action, we have the following:

\begin{lemma}\label{lem:H_0.gpd}
Let $h=(h^\Gamma,h^E)$ be a multiplicative $\Rm$-action on
$\Gamma\toto E$. Then its fixed points define an embedded Lie
subgroupoid $h^\Gamma_0(\Gamma)\toto h^E_0(E)$ of $\Gamma\toto E$.
Moreover, this Lie subgroupoid is source-simply-connected whenever
$\Gamma$ is.
\end{lemma}
\begin{proof}
Write $G=h^\Gamma_0(\Gamma)$ and $M=h^E_0(E)$. Since $h^\Gamma$ and
$h^E$ are $\Rm$-actions, we have that $G\subset\Gamma$ and $M\subset
E$ are embedded submanifolds. Using that $(h^\Gamma_0,h^E_0)$ is a
groupoid map, one may verify that $G$ and $M$ define a set-theoretic
subgroupoid of $\Gamma\toto E$.  It remains to check that the source
map of $G\toto M$ is a submersion (cf. Section~\ref{subsec:subgrp}),
which follows from the identity $\sour_G\circ h^\Gamma_0 =
h^E_0\circ \sour_\Gamma$. As for the second assertion in the lemma,
for each $x\in M$ the map $h^\Gamma_0$ defines a retraction
$\sour_\Gamma^{-1}(x) \to \sour_G^{-1}(x)$, so it induces a
surjective map at the level of fundamental groups.
\end{proof}


The next theorem characterizes VB-groupoids by means of regular
actions. We will provide a direct proof of this result now, and
discuss it from a broader viewpoint in the next section.

\begin{theorem}\label{thm:VB-gpd}
There is a one-to-one correspondence between VB-groupoids with total
groupoid $\Gamma \toto E$ and regular multiplicative actions
$\Rm\action(\Gamma \toto E)$.
\end{theorem}
\begin{proof}
Clearly every VB-groupoid has an underlying multiplicative
$\Rm$-action that is regular. Conversely, if we start with a regular
multiplicative action $h:\Rm\action (\Gamma\toto E)$, and we write
$G=h^\Gamma_0(\Gamma)$ and $M=h^E_0(E)$, then by
Thm.~\ref{thm:vertical.lift} and Lemma~\ref{lem:H_0.gpd} we obtain a
diagram of Lie groupoids and vector bundles as in \eqref{eq:VBgrd}.
It only remains to check the compatibility between these structures.
The fact that $h_\lambda$ is a groupoid map for every $\lambda$
implies that the structural maps of the groupoid structure are
equivariant, hence maps of vector bundles, and the result follows.
\end{proof}


Using this characterization we see that a Lie subgroupoid of a
VB-groupoid that is invariant for the $\Rm$-action is also a
VB-groupoid, and that a VB-groupoid map is the same as a groupoid
map between the total Lie groupoids that is $\Rm$-equivariant. It is
also clear that the direct product of VB-groupoids is a VB-groupoid:
for VB-groupoids $\Gamma_i\toto E_i$ over $G_i\toto M_i$, $i=1,2$,
the direct product of the $\Rm$-actions on the groupoid
$\Gamma_1\times \Gamma_2 \toto E_1\times E_2$ makes it into a
VB-groupoid over $G_1\times G_2\toto M_1\times M_2$.

Next we list a few other useful consequences of
Thm.~\ref{thm:VB-gpd}.

\begin{corollary}\label{cor:VB-gpd}
\
\begin{enumerate}[(a)]
 \item Let $\Gamma_1$, $\Gamma_2$ and $\Gamma$ be VB-groupoids, and let
$\Phi_i: \Gamma_i\to \Gamma$, $i=1,2$, be VB-groupoid maps forming a
good pair
 (cf. Lemma~\ref{lemma:good.pairs.groupoids}). Then their fibred product is a VB-groupoid.

 \item Given a VB-groupoid $\Gamma\toto E$ over $G\toto M$ and a Lie-groupoid map
 $(\Phi,\phi): (\tilde{G}\toto \tilde{M})\to (G\toto M)$, the fibred
 product
$$
\xymatrix{(\Phi^*\Gamma\toto \phi^*E) \ar[d]\ar[r] & (\Gamma\toto E) \ar[d]^{}\\
(\tilde{G}\toto \tilde{M}) \ar[r]^{(\Phi,\phi)}&(G\toto M)}
$$
endows the pullback vector bundles $\Phi^*\Gamma$ and $\phi^*E$ with
a VB-groupoid structure over $\tilde{G}\toto  \tilde{M}$.
\end{enumerate}
\end{corollary}

\begin{proof}


We know by Prop.~\ref{prop:fp.gpd} that the fibred product of
$\Phi_i: \Gamma_i\to \Gamma$, $i=1,2$, is a Lie subgroupoid
$\Gamma_1\times_\Gamma \Gamma_2\toto E_1\times_E E_2$ of
$\Gamma_1\times \Gamma_2 \toto E_1\times E_2$. The fact that the
maps $\Phi_i$ are vector-bundle morphisms implies that this Lie
subgroupoid is $\Rm$-invariant with respect to the direct-product
action, so it inherits a VB-groupoid structure. It is also evident
that the projections $\Gamma_1\times_\Gamma \Gamma_2 \to \Gamma_i$,
$i=1,2$, are VB-groupoid maps.

Finally, $(b)$ follows from $(a)$, since both
\smash{$\tilde{G}\xto\Phi G$} and the projection $\Gamma\to G$ are
VB-groupoid maps with respect to the VB-groupoids over $G$ and
$\tilde{G}$ with zero fibers, and the projection is a submersion,
hence transverse to any map, thus ensuring the good pair property.
\end{proof}



\begin{remark}\label{rmk:kergrp}
For a VB-groupoid map $(\Phi,\phi): (\Gamma\toto E) \to
(\tilde{\Gamma}\toto \tilde{E})$, assume that the underlying
vector-bundle map $\Phi: (\Gamma\to G) \to (\tilde{\Gamma}\to
\tilde{G})$ has constant rank. Then  $\phi: (E\to M)\to
(\tilde{E}\to \tilde{M})$ also has constant rank, and one can
directly check that $\ker(\Phi)$ and $\ker(\phi)$ define a
VB-groupoid $\ker(\Phi) \toto \ker(\phi)$ over $G\toto M$, cf.
\cite[App. C]{lm}. (Note that this VB-groupoid may be seen,
following Cor.~\ref{cor:VB-gpd}(a), as the fibred product of the
good pair given by $\Phi$ and the zero section $\tilde{G}\to
\tilde{\Gamma}$.)
\end{remark}


\begin{remark}
Just as the usual category of vector bundles over a manifold, the
category $\mathrm{VB}(G\toto M)$ of VB-groupoids over $G\toto M$ is
an additive category (it is equipped with linear structures on
hom-sets, zero objects and direct sums). Moreover, we can take
kernels and cokernels of VB-groupoid maps that have constant rank
and consider short exact sequences of VB-groupoids. (The
construction of kernels is discussed in the previous remark, and for
cokernels one may compute the kernels of the dual maps, and dualize
again.)
\end{remark}

\begin{remark}\label{rem:pbgrp}
Given a Lie-groupoid map $\Phi:(\tilde{G}\toto \tilde{M})\to(G\toto
M)$, the pullback construction from Cor.~\ref{cor:VB-gpd}(b) defines
a {\em base-change} functor
\begin{equation}\label{eq:pbgp}
\mathrm{VB}(G\toto M)\xto{\Phi^*}\mathrm{VB}(\tilde{G}\toto
\tilde{M}),
\end{equation}
which preserves short exact sequences and duals (i.e., the natural
vector-bundle identifications $(\Phi^* \Gamma)^* \cong \Phi^*
\Gamma^*$ are isomorphisms of VB-groupoids). Moreover, the property
that any vector-bundle map $\Psi: (\Gamma\to G) \to
(\tilde{\Gamma}\to \tilde{G})$ induces a vector-bundle map
$\Gamma\to (\Psi|_G)^* \tilde{\Gamma}$ over the identity $G\to G$
also extends: when $\Gamma$ and $\tilde{\Gamma}$ are VB-groupoids
and $\Psi$ is a VB-groupoid map, so is this induced map.
\end{remark}


\subsection{VB-algebroids}\label{sec:IMact}


A {\em VB-algebroid} \cite{mac-double2,GM}  consists of a double
vector bundle $(\Omega,E,A,M)$ equipped with a Lie-algebroid
structure $\Omega\then E$ that is compatible with the second
vector-bundle structure on $\Omega$ in the following sense: we
require the linear Poisson structure on $\Omega^*_E$, corresponding
to $\Omega\then E$, to be also linear with respect to the fibration
$\Omega^*_E\to C^*$ (see \cite[Sec.~3.4]{GM}); in other words, it is
a double linear Poisson structure on the horizontal dual,
\begin{equation}\label{eq:VBalgdual}
\begin{matrix}
\Omega_E^* & \to & E \\ \downarrow & & \downarrow\\ C^* & \to & M.
\end{matrix}
\end{equation}

One can check that a double linear Poisson structure on $\Omega^*_E$
induces a unique Lie-algebroid structure on $A\to M$ so that the
projection $q^\Omega_A$ (and also the zero section $0^\Omega_A$) is
a Lie-algebroid map, see \cite[Sec.~3]{GM} (cf.
Lemma~\ref{lem:H0_alg} and Thm.~\ref{thm:VB-alg} below). For this
reason, we depict a VB-algebroid as follows:
\begin{equation}\label{eq:VBalg}
\begin{matrix}
\Omega & \then & E\\
\downarrow & & \downarrow\\
A& \then & M.
\end{matrix}
\end{equation}
We refer to $\Omega \then E$ as the {\em total algebroid}, while
$A\then M$ is the {\em base algebroid}, and we say that $\Omega
\then E$  is a VB-algebroid over $A\then M$.  The {\em core} of a
VB-algebroid is that of the underlying double vector bundle.

Other formulations of VB-algebroids and their equivalences can be
found in \cite{GM}.


A {\em VB-algebroid map} is a Lie-algebroid morphism
$\Phi:(\Omega\then E)\to(\tilde{\Omega}\then \tilde{E})$ that is
also linear, i.e., it defines a vector-bundle map $(\Omega\to A)\to
(\tilde{\Omega}\to \tilde{A})$. As a consequence, one can check that
$\Phi$ restricts to a Lie-algebroid map $ (A\then M)\to
(\tilde{A}\then \tilde{M})$. As in the case of VB-groupoids, a
VB-algebroid map determines four smooth maps relating each of the
involved manifolds and preserving all structures. We denote by
$\mathrm{VB}(A\then M)$ the category of VB-algebroids over $A\then
M$, with VB-groupoid maps covering the identity as morphisms.


\begin{example}[Tangent algebroid]\label{ex:tgVB}
If $A\then M$ is a Lie algebroid, then it induces a Lie-algebroid
structure on $TA\to TM$, often referred to as the {\em tangent
prolongation} of $A$; in this example, the anchor and the bracket
between tangent sections are obtained by differentiating the
corresponding structures in $A$, see \cite[\S 9.7]{Mac-book}. The
Lie algebroid $TA\then TM$ is a VB-algebroid over $A\then M$.
\end{example}

\begin{example}[Cotangent algebroid]\label{ex:ctgVB}
For a Lie algebroid $A\then M$, the corresponding linear Poisson
structure on $A^*$ induces a Lie-algebroid structure $T^*A^* \then
A^*$, which is a VB-algebroid over $A\then M$ (c.f.
Remark~\ref{rem:poisson}). Using the reversal isomorphism
\eqref{eq:R}, we then obtain a VB-algebroid $T^*A\then A^*$ over
$A\then M$.
\end{example}


\begin{example}[Representations]\label{ex:reps}
Any representation of a Lie algebroid $A\then M$ gives rise to a
VB-algebroid over $A$, for which the ranks of $\Omega\to A$ and
$E\to M$ are equal; in fact there is a one-to-one correspondence
between representations of $A$ and VB-algebroids with this
additional property. Just as  VB-groupoids (c.f.
Remark~\ref{rmk:ruth}), VB-algebroids are thought of as generalized
representations, since they are similarly related to 2-term
representations up to homotopy, see \cite{GM}.
\end{example}


Given a VB-algebroid \eqref{eq:VBalg}, while its horizontal dual
\eqref{eq:VBalgdual} carries a double linear Poisson structure, its
vertical dual is again a VB-algebroid \cite[Sec.~3.4]{GM}:
\begin{equation}\label{eq:dualalg}
\begin{matrix}
\Omega^*_A & \then & C^*\\
\downarrow & & \downarrow\\
A& \then & M.
\end{matrix}
\end{equation}
Indeed, using the linearity of the Poisson structure relative to
$\Omega_E^*\to C^*$, we see that $(\Omega_E^*)^*_{C^*} \then C^*$ is a
VB-algebroid over $A\then M$. The VB-algebroid \eqref{eq:dualalg} is
defined by means of the isomorphism $Z$ in \eqref{eq:Z}, and we
refer to it as the {\em dual VB-algebroid} of $\Omega\then E$. As in
the case of VB-groupoids, duality defines an involution in the
category $\mathrm{VB}(A\then M)$.

\begin{example}
The tangent and cotangent Lie algebroids, see Examples~\ref{ex:tgVB}
and \ref{ex:ctgVB}, are dual VB-algebroids in the sense just
described. In other words, $Z: (T^*A \then A^*) \to
((TA)^*_{TM})^*_{A^*}\then A^*)$ is a Lie-algebroid map. This can be
verified by noticing that the composition $R\circ Z^{-1}$, where $R$
is the reversal isomorphism (see \cite[Prop.~9.3.2 \& \S
9.5]{Mac-book}), is a Lie-algebroid map (by
\cite[Thm.~10.3.14]{Mac-book}); since $R$ is a Lie-algebroid map by
definition, so is $Z$.
\end{example}



\subsection{Characterization as IM actions}


We now discuss the Lie-algebroid version of Theorem
\ref{thm:VB-gpd}. We start with the infinitesimal counterpart of
multiplicative actions:

\begin{definition}
An $\Rm$-action $h$ on the total space of a Lie algebroid
$\Omega\then E$ is called {\em infinitesimally multiplicative} (or
simply {\em IM}) if each $h_\lambda : \Omega\to \Omega$ defines a
Lie-algebroid map.
\end{definition}

As in the case of VB-groupoids, the restriction of $h$ to $E$ is an
$\Rm$-action, which is regular if $h$ is, and we can think of $h$ as
a pair of actions $(h^\Omega,h^E)$, on $\Omega$ and $E$.


We now consider fixed points of IM-actions. The discussion parallels
Lemma~\ref{lem:H_0.gpd}, though its direct proof cannot be adapted
to Lie algebroids. So we provide an alternative argument using
fibred products, which works in both contexts.

\begin{lemma}\label{lem:H0_alg}
If $h=(h^\Omega,h^E)$ is an IM-action on $\Omega\then E$, then its
fixed points define an embedded Lie subalgebroid
$h_0^\Omega(\Omega)\then h^E_0(E)$ of $\Omega\then E$.
\end{lemma}

\begin{proof}
Write $A= h^\Omega_0(\Omega)$ and $M=h^E_0(E)$. Since $h^\Omega$ and
$h^E$ are $\Rm$-actions,  $A\subset \Omega$ and $M\subset E$ are
embedded submanifolds. One then verifies that the following pair of
Lie-algebroid maps is good in the sense of the Appendix (see
Lemma~\ref{lemma:good.pairs.alg}):
\begin{equation}\label{eq:fixed}
\begin{matrix}
\xymatrix{
& (\Omega\then E) \ar[d]^{(h^\Omega_0,h^E_0)\times (\id_\Omega,\id_E)} \\
(\Omega\then E)  \ar[r]_(.4){\Delta} & (\Omega\then
E)\times(\Omega\then E),}
\end{matrix}
\end{equation}
where $\Delta$ is the diagonal map. It follows from
Prop.~\ref{prop:fp.alg} that their algebroid-theoretic fibred
product is well-defined. The natural identification of this fibred
product with $A\to M$ makes it into a Lie subalgebroid of
$\Omega\then E$.
\end{proof}


The following result characterizes VB-algebroids in terms of IM-actions:

\begin{theorem}
\label{thm:VB-alg} There is a one-to-one correspondence between
VB-algebroids with total Lie algebroid $\Omega \then E$ and regular
IM actions $\Rm\action(\Omega \then E)$.
\end{theorem}
\begin{proof}
We know that a regular action $h$ on $\Omega\to E$ by vector-bundle
maps is the same as a double vector bundle structure
$(\Omega,E,A,M)$, where $A=h_0^\Omega(\Omega)$ and $M=h_0^E(E)$ (cf.
Prop.~\ref{prop:DVB}). So what we need to show is that $h$ acts by
Lie-algebroid maps if and only if the compatibility condition
defining VB-algebroids is fulfilled.

Let $k$ be the action corresponding to $\Omega\to E$, and let
$\bar{h}, \bar{k}$  be the $\Rm$-actions associated to
$\Omega^*_E\to C^*$ and $\Omega^*_E\to E$, respectively. For each
$\lambda\neq 0$, by Prop.~\ref{tildeh} we know that $\bar{h}_\lambda
= (h_{\lambda^{-1}})^*\circ \bar{k}_\lambda$. Letting $\pi$ denote
the Poisson bivector field on $\Omega_E^*$ corresponding to the
Lie-algebroid structure $\Omega\then E$, we consider the following
diagram:
$$
\xymatrix{ (\Omega^*_E,\pi) \ar[rd]_{\bar{h}_\lambda}
\ar[rr]^{\bar{k}_\lambda}& &
(\Omega^*_E,\lambda\pi) \ar[dl]^{h_{\lambda^{-1}}^* }\\
& (\Omega^*_E,\lambda\pi).}
$$
First, we notice that it is commutative as a diagram of smooth maps.
In addition, the map $\bar{k}_\lambda$ is Poisson as a result of the
linearity of $\pi$ over $E$ (see Prop.~\ref{linear}). It then
follows that the left arrow is a Poisson map if and only if so is
the right arrow.

Note that
$\bar{h}_\lambda:(\Omega^*_E,\pi)\to(\Omega^*_E,\lambda\pi)$ is a
Poisson map for every $\lambda\neq 0$ if and only if the Poisson
structure $\pi$ on $\Omega^*_E$ is linear with respect to
$\Omega_E^*\to C^*$ (cf. Prop.~\ref{linear}), and this is by
definition the same as $\Omega\then E$ being a VB-algebroid. On the
other hand, the map
$h_{\lambda^{-1}}^*:(\Omega^*_E,\lambda\pi)\to(\Omega^*_E,\lambda\pi)$
is Poisson if and only if $h_{{\lambda}^{-1}}^*$ is a Poisson
automorphism of $(\Omega_E^*,\pi)$, or equivalently,
$h_{{\lambda}^{-1}}$ is an algebroid map. Note that if this holds
for every $\lambda\neq 0$ then it is also true for $\lambda=0$, by
passing to the limit.
\end{proof}

As previously mentioned, on a VB-algebroid the base Lie algebroid is
determined by the total algebroid. Note that Lemma~\ref{lem:H0_alg}
explains this fact from the point of view of $\Rm$-actions.


Just as in the discussion for VB-groupoids, we conclude that a Lie
subalgebroid of a VB-algebroid that is invariant under the
$\Rm$-action is a VB-algebroid itself. It is also clear that a
VB-algebroid map is a Lie-algebroid map between the total Lie
algebroids that is $\Rm$-equivariant, and that the direct product of
VB-algebroids is a VB-algebroid. In the remainder of this section we
collect other consequences of Thm.~\ref{thm:VB-alg}, which are
parallel to those for groupoids presented after
Thm.~\ref{thm:VB-gpd}.

\begin{corollary}\
\label{cor:VB-alg}
\begin{enumerate}[(a)]
\item Given VB-algebroids $\Omega_1$, $\Omega_2$ and $\Omega$ and
VB-algebroid maps $\Omega_i\to \Omega$, $i=1,2$, if the maps form a
good pair (cf. Lemma~\ref{lemma:good.pairs.alg}), then their fibred
product is a VB-algebroid.
\item
Given a VB-algebroid $\Omega\then E$ over $A\then M$ and a
Lie-algebroid map $(\Phi,\phi): (\tilde{A}\then \tilde{M})\to
(A\then M)$, the fibred product
$$
\xymatrix{(\Phi^*\Omega\then \phi^*E) \ar[d]\ar[r] & (\Omega\then E) \ar[d]^{}\\
(\tilde{A}\then \tilde{M}) \ar[r]^{(\Phi,\phi)}&(A\then M)}
$$
endows the pullback vector bundles $\Phi^*\Omega$ and $\phi^*E$ with
a VB-algebroid structure over $\tilde{A}\then \tilde{M}$.

\end{enumerate}
\end{corollary}

\begin{proof}
The proofs are completely analogous to those in
Corollary~\ref{cor:VB-gpd}, but now making use of
Prop.~\ref{prop:fp.alg} rather than Prop.~\ref{prop:fp.gpd}.
\end{proof}


\begin{remark}\label{rmk:keralg}
The category $\mathrm{VB}(A\then M)$ is additive, and one can also
consider kernels and co-kernels: given a VB-algebroid map
$(\Phi,\phi): (\Omega\then E)\to (\tilde{\Omega}\then \tilde{E})$,
if it has constant rank as a map of vector bundles $\Phi: (\Omega\to
A)\to (\tilde{\Omega}\to \tilde{A})$ (the same automatically holds
for $\phi: (E\to M)\to (\tilde{E}\to \tilde{M})$), then its kernel
naturally inherits the structure of a VB-algebroid $\ker(\Phi)\then
\ker(\phi)$ over $A\then M$. This may not be as direct to check as
in the groupoid case, but one may use the same argument sketched in
the end of Remark~\ref{rmk:kergrp}: the map $\Phi$ and the zero
section $\tilde{A}\to \tilde{\Omega}$ form a good pair of
VB-algebroid maps (considering the zero vector bundle over
$\tilde{A}$), so by Cor.~\ref{cor:VB-alg}(a) one can realize the
kernel of $\Phi$ as the fibred-product VB-algebroid
\begin{equation}\label{eq:kerVBgr}
\begin{matrix}
\xymatrix{ (\ker(\Phi)\then\ker(\phi)) \ar[r] \ar[d] & (\Omega\then E) \ar[d]^{(\Phi,\phi)} \\
(\tilde{A}\then
\tilde{M})\ar[r]_{(0_{\tilde{\Omega}},0_{\tilde{E}})} &
(\tilde{\Omega}\then \tilde{E}).}
\end{matrix}
\end{equation}
The construction of cokernels follows from duality.
\end{remark}

\begin{remark}\label{rem:pbalg}
Through the pullback construction of Cor.~\ref{cor:VB-alg}(b), each
Lie-algebroid map $\Phi: (\tilde{A}\then \tilde{M})\to (A\then M)$
gives rise to a base-change functor
\begin{equation}\label{eq:pbalg}
\mathrm{VB}(A\then M)\xto{\Phi^*}\mathrm{VB}(\tilde{A}\then
\tilde{M})
\end{equation}
preserving short exact sequences and duals. If $\Phi: \Omega\to
\tilde{\Omega}$ is a VB-algebroid map, then the induced
vector-bundle map $\Omega \to (\Phi|_A)^*\tilde{\Omega}$, covering
the identity map on $A\to M$, is a also VB-algebroid map.
\end{remark}


\section{Lie theory for vector bundles}\label{sec:Lie}


A differentiation procedure (see e.g. \cite{Mac-book,mm}) gives rise
to a functor
\begin{equation}\label{eq:lieF}
\mathrm{Lie\ Groupoids}\xto{\mathrm{Lie}} \mathrm{Lie\ Algebroids}.
\end{equation}
It is well known that this Lie functor is not an equivalence of
categories, so a perfect translation between the global and
infinitesimal pictures is not always possible.


For a Lie groupoid $G$ we use the notation $A_G = \mathrm{Lie}(G)$
and say that $G$ \emph{integrates} $A$. It is a fact that not every
Lie algebroid comes from a Lie groupoid, see \cite{CF} for a
discussion of the integrability problem.

For a morphism $\Phi:G_1\to G_2$ we often write $\Phi' =
\mathrm{Lie}(\Phi):A_{G_1}\to A_{G_2}$ to simplify the notation.
Upon an additional topological assumption, namely if $G_1$ is a
source-simply connected, the Lie functor sets a bijection between
groupoid maps $G_1\to G_2$ and algebroid maps $A_{G_1}\to A_{G_2}$.
This is the content of {\em Lie's second theorem} (see e.g.
\cite[Sec.~6.3]{mm}), that will be used recurrently in this paper.

\begin{remark}\label{rem:sconnect}
If $G_1$ is just source-connected, we still have injectivity: if
$\Phi, \Psi: G_1 \to G_2$ are groupoid maps such that $\Phi'=\Psi'$,
then necessarily $\Phi=\Psi$.
\end{remark}

In this section we use Theorems~\ref{thm:VB-gpd} and
\ref{thm:VB-alg} to explain how VB-groupoids and VB-algebroids are
related by differentiation and integration.

\subsection{Differentiation of VB-groupoids}


In order to relate VB-groupoids and VB-algebroids by the Lie
functor, it will be convenient to consider the following alternative
formulations of multiplicative and IM actions.

Denoting by $\R\toto\R$ the unit groupoid of the real line $\R$, one
can directly see that a multiplicative action $h: \Rm \action
(\Gamma\toto E)$ is the same as a Lie groupoid map
\begin{equation}\label{eq:href}
h:(\Gamma\toto E)\times(\R\toto\R)\to(\Gamma\toto E)
\end{equation}
satisfying $h_1=\id$ and $h_\lambda
h_{\lambda'}=h_{\lambda\lambda'}$, for all $\lambda,\lambda'\in\R$.
Analogously, if $0_{\R}\then\R$ is the zero Lie algebroid over the
real line, then an IM action $h:\Rm\action(\Omega\then E)$ is
equivalent to a Lie-algebroid map
\begin{equation}\label{eq:hrefalg}
h:(\Omega\then E)\times(0_{\R}\then\R)\to(\Omega\then E)
\end{equation}
satisfying $h_1=\id$ and $h_\lambda h_{\lambda'}=h_{\lambda\lambda'}$, for all $\lambda,\lambda'\in\R$.

\begin{proposition}
\label{prop:diff.action}
Let $h: \Rm\action (\Gamma\toto E)$ be a multiplicative action.
Then:
\begin{enumerate}[(a)]
\item The map $A_\Gamma\times \mathbb{R}\to A_\Gamma$ given by
$(a,\lambda) \mapsto (h_\lambda)'(a)$ defines an IM-action $h'$ on
$A_\Gamma\then E$.

\item If $G\toto M$ denotes the Lie subgroupoid of $\Gamma\toto E$ given
by the fixed points of $h$, then the fixed points of $h'$ are
identified with $A_G\then M$.
\item If $h$ is a regular action, then so is $h'$.
\end{enumerate}
\end{proposition}
\begin{proof}
The Lie functor \eqref{eq:lieF} preserves products and maps $\R\toto
\R$ to $0_\R\then \R$. So viewing the action $h$ as a groupoid map
as in \eqref{eq:href}, it is immediate that by differentiation we
obtain a Lie-algebroid map $h': A_\Gamma\times \mathbb{R}\to
A_\Gamma$. One can also directly check that $h'$ satisfies
$$
(h')_\lambda = (h_\lambda)'
$$
for all $\lambda\in \R$, from where we see that $h'_1=\id$,
$h'_\lambda h'_\mu=h'_{\lambda\mu}$. So $(a)$ follows.

Regarding $(b)$, we can express the fixed points of $h$, i.e., the
image of $h_0$, as the good fibred product between the map
$(h_0,\id_\Gamma): \Gamma\to \Gamma\times \Gamma$ and the diagonal
$\Delta_\Gamma: \Gamma\to \Gamma\times \Gamma$ (c.f.
\eqref{eq:fixed}). The same holds for the fixed points of $h'$, now
considering the maps $(h'_0,\id_{A_\Gamma}): A_\Gamma\to
A_\Gamma\times A_\Gamma$ and the diagonal $\Delta_{A_\Gamma}$. Note
that these maps on $A_\Gamma$ correspond to the maps previously
defined on $\Gamma$ by the Lie functor. The conclusion follows from
the fact thet the Lie functor preserves fibred products, as shown in
Prop.~\ref{prop:diff.fp}.

Finally, $(c)$ holds because $h'$ is a restriction of the tangent
lift action $\dd h:\Rm\action T\Gamma$, which is regular if $h$ is.
\end{proof}


The previous proposition, together with our characterizations of
VB-groupoids and VB-algebroids in Theorems~\ref{thm:VB-gpd} and
\ref{thm:VB-alg}, lead to:

\begin{corollary}
\label{cor:diff.VB} If $\Gamma\toto E$ is a VB-groupoid over $G\toto
M$, then $A_{\Gamma}\then E$ inherits a VB-algebroid structure over
$A_{G}\then M$.
\end{corollary}

\begin{remark}
By viewing vector bundles as particular cases of Lie groupoids and
Lie algebroids, one may view VB-groupoids (resp. VB-algebroids) as
special types of double Lie groupoids (resp. LA-groupoids);
Corollary~\ref{cor:diff.VB} then also follows from the fact that
double Lie groupoids can be differentiated to LA-groupoids
\cite{mac-double1}.
\end{remark}


Corollary~\ref{cor:diff.VB} is part of a more general observation:
since VB-groupoid and VB-algebroid maps are characterized by
$\Rm$-equivariance, it is a direct verification that the Lie functor
\eqref{eq:lieF} restricts to a functor
\begin{equation}\label{eq:lieFVB}
\mathrm{VB}(G\toto M)\xto{\Lie} \mathrm{VB}(A_G\then M).
\end{equation}

\begin{remark}\label{rmk:identifications}\
The functor \eqref{eq:lieFVB} satisfies the following natural
properties, that we explicitly  describe for later use:
\begin{enumerate}[(a)]
\item
It commutes with the pullback functors defined in \eqref{eq:pbgp}
and \eqref{eq:pbalg}: If $\Phi: G_1 \to G_2$ is a Lie-groupoid map
and $\Gamma$ is a VB-groupoid over $G_2$, then the identification in
Prop.~\ref{prop:diff.fp} yields an isomorphism of VB-algebroids over
$A_{G_1}$,
$$
r_\Gamma: A_{\Phi^*\Gamma}\to (\Phi')^*A_{\Gamma}.
$$
which is natural, namely $r_{\Gamma_2}\circ (\Phi^*(\Psi))' =
((\Phi')^*(\Psi')) \circ r_{\Gamma_1}$ for any $\Psi: \Gamma_1\to
\Gamma_2$.

\item
It preserves duals as described in \eqref{eq:dualgp} and
\eqref{eq:dualalg}: given a VB-groupoid $\Gamma$ over $G$, by
differentiating the canonical pairing $\Gamma^*\times_G\Gamma\to\R$
we obtain a natural isomorphism of VB-algebroids
(c.f.\cite[Thm~11.5.12]{Mac-book})
$$
i_\Gamma:A_{\Gamma^*} \to (A_\Gamma)^*_{A_G}.
$$
For a VB-groupoid map $\Psi:\Gamma_1\to\Gamma_2$ over $G\toto M$ we have
$(\Psi')^*\circ i_{\Gamma_2} = i_{\Gamma_1}\circ (\Psi^*)'$.

\item
It maps tangent VB-groupoids to tangent VB-algebroids upon the
identification given by restriction of the natural involution of the
double tangent bundle (see \cite[Thm.~9.7.5]{Mac-book}),
$$
j_G:TA_G\to A_{TG}.
$$
This fact, combined with the previous item (b), shows that the Lie
functor \eqref{eq:lieFVB} also maps  cotangent VB-groupoids to
cotangent VB-algebroids, via
$$
\theta_G=j_G^*i_{TG}:A_{T^*G}\to T^*A_G.
$$
Given a Lie groupoid map $\Phi:G_1\to G_2$, the naturality of these
identifications is expressed by the following equations:
\begin{equation}\label{eq:Tj}
(\dd \Phi)'\circ j_{G_1} = j_{G_2}\circ \dd (\Phi'), \qquad \qquad
(\dd (\Phi'))^*\circ \theta_{G_2}=\theta_{G_1} \circ ((\dd
\Phi)^*)'.
\end{equation}

\item
The Lie functor preserves short exact sequences, in particular
kernels and cokernels. One can check that it preserves kernels by
expressing the kernel of a map $\Phi: \Gamma_1\to \Gamma_2$ as the
good fibred product between the map itself and the zero section of
$\Gamma_2$, and using Prop.~\ref{prop:diff.fp}. The fact that it
preserves cokernels follows, for instance, from this property for
kernels and the behavior under duality.
\end{enumerate}
\end{remark}


\subsection{The vertical lift for multiplicative and IM actions}


In order to study the integration of VB-algebroids, i.e., the
inverse procedure to Corollary~\ref{cor:diff.VB}, it will be
convenient to reformulate Theorems~\ref{thm:VB-gpd} and
\ref{thm:VB-alg} as a more general functorial construction, building
on Theorem~\ref{thm:vertical.lift}.


Given a multiplicative action $h:\Rm\action(\Gamma\toto E)$, not
necessarily regular, we know that its fixed points define a Lie
subgroupoid $G\toto M$ (by Lemma~\ref{lem:H_0.gpd}). The action has
an associated vertical bundle $V_h\Gamma\to G$ (see \eqref{eq:VE}),
while the restriction of $h$ to $E$ gives rise to the vertical
bundle $V_hE \to M$. A key observation is that $V_h\Gamma$ is a Lie
groupoid over $V_hE$, and $V_h\Gamma\toto V_hE$ is a VB-groupoid
over $G\toto M$; these facts follow from the constructions of
kernels and pullbacks in Remark~\ref{rmk:kergrp} and
Cor.~\ref{cor:VB-gpd}(b), since $V_h\Gamma$ is obtained by
restricting to $G$ the kernel of the VB-groupoid map $Th: T\Gamma\to
T\Gamma$. We refer to the VB-groupoid
\begin{equation}\label{eq:vlgp}
V_h\Gamma \toto V_hE
\end{equation}
over the fixed points $G\toto M$ as the {\em vertical bundle} of the
multiplicative action $h$. Note that, by construction, $V_h\Gamma$
naturally sits in the tangent VB-groupoid $T\Gamma\toto TE$ as a
VB-subgroupoid.

The following result offers a more general viewpoint to Theorem~\ref{thm:VB-gpd}.

\begin{proposition}
\label{prop:vl.VB-gpd} Let $\Gamma\toto E$ be a Lie groupoid endowed
with a multiplicative action $h:\Rm\action(\Gamma\toto E)$. Then:
\begin{enumerate}[(a)]
 \item The vertical lift $\V_h: (\Gamma\toto E)\to (V_h\Gamma\toto V_hE)$
is a Lie-groupoid morphism which is $\Rm$-equivariant;
 \item
The action $h$ is regular if and only if $\V_h$ is an isomorphism
onto the vertical bundle $V_h\Gamma$, in which case $\Gamma\toto E$
inherits the structure of a VB-groupoid.
\end{enumerate}
\end{proposition}
\begin{proof}
For $(a)$, since $V_h\Gamma\toto V_hE$ is a Lie subgroupoid of
$T\Gamma\toto TE$, and in light of Remark~\ref{rem:factorization},
we just need to show that the composition $\dd h\circ l$ is a
Lie-groupoid map, where $l:\Gamma\to T\Gamma\times T\R$ is defined
as in \eqref{eq:jdiag}. The tangent construction is functorial, so
$\dd h$ is a Lie-groupoid map (since so is \eqref{eq:href}).
Moreover, each component of $l$ is a Lie-groupoid map, since the
first one is the zero section of the tangent VB-groupoid of
$\Gamma\toto E$, while the second is a constant map into the unit
groupoid $T\R\toto T\R$.

The assertion in $(b)$ is an immediate consequence of
Theorem~\ref{thm:vertical.lift}.
\end{proof}


There is an analogous result for IM actions on Lie algebroids. Given
an IM action $h:\Rm\action(\Omega\then E)$, not necessarily regular,
denote by $A\then M$ the Lie algebroid defined by the fixed points
of $h$ (c.f. Lemma~\ref{lem:H0_alg}). We have a natural VB-algebroid
$$
V_h\Omega\then V_hE
$$
over $A\then M$. As in the case of groupoids, it is a
VB-subalgebroid of $T\Omega\then TE$. These results are direct
consequences of our observations on pullbacks and kernels in
Cor~\ref{cor:VB-alg}(b) and Remark~\ref{rmk:keralg}. We refer to the
VB-algebroid $V_h\Omega\then V_hE$  as the {\em vertical bundle} of
the IM action $h$. Reasoning as in Proposition~\ref{prop:vl.VB-gpd},
we obtain:

\begin{proposition}
\label{prop:vl.VB-alg} Let $\Omega\then E$ be a Lie algebroid and
$h:\Rm\action(\Omega\then E)$ an IM action. Then:
\begin{enumerate}[(a)]
 \item
The vertical lift $\V_h: (\Omega\then E)\to (V_h\Omega\then V_hE)$
is Lie-algebroid morphism which is $\Rm$-equivariant;
 \item
The action $h$ is regular if and only if $\V_h$ is an isomorphism
onto the vertical bundle $V_h\Omega$. In this case $\Omega\then E$
inherits the structure of a VB-algebroid.
\end{enumerate}
\end{proposition}


As in Remark~\ref{rem:regular}, the last two propositions define
{\it regularization} functors from the categories of multiplicative
actions and IM actions to the categories of VB-groupoids and
VB-algebroids, respectively. The vertical bundle is the regular
object associated to an action in each case.


The following result clarifies the relation between vertical
lifts and the Lie functor.

\begin{proposition}
\label{prop:diff.vl} Let $h:\Rm\action(\Gamma\toto E)$ be a
multiplicative action, and let $h':\Rm\action(\Omega\then E)$ be the
corresponding IM action. Then the canonical isomorphism $j_\Gamma:
TA_\Gamma\to A_{T\Gamma}$ restricts to an isomorphism $ V_{h'}
A_\Gamma \xto\sim A_{V_h\Gamma}$ so that $j_\Gamma\circ \V_{h'} =
(\V_h)'$:
$$
\xymatrix{
A_\Gamma \ar[r]^{\V_{h'}} \ar[dr]_{(\V_h)'} & V_{h'}A_\Gamma \ar[d]^{j_\Gamma}_{\wr}\\
 & A_{V_h\Gamma}.
}
$$
\end{proposition}

\begin{proof}
Let us view the action $h$ as in \eqref{eq:href} and consider the
factorization of $\V_{h}$ as in Remark~\ref{rem:factorization}; we
write $\V_h = \dd h \circ l_\Gamma$, recalling that $l_\Gamma:
\Gamma \to T\Gamma\times T\R$ is a Lie-groupoid map. There is an
analogous factorization associated with $h'$, that we write as
$\V_{h'}=\dd (h') \circ l_A$, and $l_A: A_{\Gamma}\to TA_\Gamma \times
T0_\R$ is a Lie-algebroid map. By considering each component of the
map $l_\Gamma : \Gamma \to T\Gamma \times T\R$ (and recalling that
both tangent and Lie functors respect direct products), one may
directly check that
$$
l_\Gamma' = j_{\Gamma\times \R} \circ l_A : A_{\Gamma}\to
A_{T\Gamma\times T\R}.
$$
By using the naturality of $j$ (cf.
Remark~\ref{rmk:identifications},(c)), we conclude that
\begin{align*}
(\V_h)' & = (\dd h)' \circ l_\Gamma' = (\dd h)' \circ
j_{\Gamma\times \R} \circ
l_A\\
&= j_{\Gamma} \circ \dd (h') \circ l_A = j_{\Gamma} \circ
\V_{h'}.\qedhere
\end{align*}
\end{proof}


\subsection{Integration}



Assuming that the total Lie algebroid of a VB-algebroid is
integrable, the issue discussed in this subsection is whether it is
integrated by a VB-groupoid. Following Theorems~\ref{thm:VB-gpd} and
\ref{thm:VB-alg}, the problem of integrating VB-algebroids can be
viewed in two steps: first integrating IM actions to multiplicative
actions, and then dealing with the additional regularity condition.


The first step, integration of IM actions, is handled by Lie's
second theorem, recalled in the beginning of the section.

\begin{lemma}\label{lem:int.action}
Let $\Gamma\toto E$ be a source-simply-connected Lie groupoid with
Lie algebroid $\Omega\then E$. Then any IM action $\tilde{h}:
\Rm\action(\Omega\then E)$ integrates
 to a multiplicative
action $h: \Rm\action(\Gamma\toto E)$, in a way such that
$(h_\lambda)'=\tilde{h}_\lambda$,  for all $\lambda$.
\end{lemma}

\begin{proof}
Viewing the IM action as a Lie-algebroid map $\tilde{h}:
\Omega\times 0_\R\to \Omega$ (c.f. \eqref{eq:hrefalg}), and since
the source-fibers of $\Gamma\times\R$ are diffeomorphic to those of
$\Gamma$, we can integrate $\tilde h$ via Lie's second theorem to
obtain a Lie-groupoid map $h: \Gamma \times \R \to \R$ (as in
\eqref{eq:href}) such that $h'=\tilde{h}$. The uniqueness of the
integration of maps implies that $(h')_\lambda =
(h_\lambda)'=\tilde{h}_\lambda$, from where the action axioms for
$h$ directly follow.
\end{proof}


The next example illustrates the relevance of the
source-simply-connectedness hypothesis in the previous lemma.

\begin{example}
Let $\R\then\ast$ be the 1-dimensional Lie algebra, viewed as a
VB-algebroid over the point $\ast\then\ast$. If we take $\Gamma=S^1$
as the Lie group integrating $\R$, then it is not possible to
integrate the action by homotheties $\tilde{h}$ as in
Lemma~\ref{lem:int.action}, since its only fixed point
$h_0(S^1)=\ast$ (see Prop.~\ref{prop:diff.action}(b)) would have to
be a retract of $S^1$, which cannot happen.
\end{example}


We now address the second step, that of regularity, by proving the
converse to Prop.~\ref{prop:diff.action}(c).

\begin{proposition}\label{prop:int.action}
Let $\Gamma\toto E$ be a Lie groupoid equipped with a multiplicative action
$h:\Rm\action(\Gamma\toto E)$, and let $h'$ the corresponding IM action on $A_\Gamma$.
If $h'$ is regular then so is $h$.
\end{proposition}
\begin{proof}
In the commutative triangle of Proposition~\ref{prop:diff.vl},
we know that $\V_{h'}$ is an isomorphism because $h'$ is regular, so
$(\V_h)': A_\Gamma \to A_{V_h\Gamma}$ is an isomorphism as well.

If we assume that $\Gamma\toto E$ is source-simply connected, then
by Lemma~\ref{lem:H_0.gpd} we know that $G = h_0(\Gamma)$ is
source-simply-connected, and hence so is $V_h\Gamma$ (since it is a
VB-groupoid over $G$, see Remark~\ref{rem:sscandmap}). It then
follows that $\V_h: \Gamma\to V_h\Gamma$ must be an isomorphism,
showing that $h$ is regular. Not assuming
source-simply-connectedness of $\Gamma$, we need a more elaborate
argument.

The key observation is that a groupoid map $\Phi: \Gamma_1\to
\Gamma_2$ is \'etale, i.e., its differential is invertible at all
points, if and only if $\Phi': A_{\Gamma_1}\to A_{\Gamma_2}$ is an
isomorphism. To see that, consider the induced VB-groupoid map $\dd
\Phi: T\Gamma_1\to T\Gamma_2$ and use Remark~\ref{rem:sscandmap}(b).
In our case, we conclude that the groupoid map $\V_h: \Gamma\to
V_h\Gamma$ is \'etale. The proof ends with the observation that this
cannot happen for the vertical lift corresponding to a non-regular
action. Indeed, if $h$ is not regular, then there exists $z
\in\Gamma$ such that $h_0(z)\neq z=h_1(z)$ and $\V_h(z)=0$. In
particular, the curve $\lambda\mapsto h_\lambda(z)$ is not constant,
so there is a point with non-zero velocity vector $X$. But
$\V_h(h_\lambda(z))=0$ for all $\lambda$ and therefore the
differential of $\V_h$ vanishes on $X$.
\end{proof}


For a VB-algebroid $\Omega\then E$ over $A\then M$, since $A$ sits
in $\Omega$ as a Lie subalgebroid, the assumption that $\Omega\then
E$ is integrable implies that so is $A\then M$, see e.g.
\cite[Prop.~6.7]{mm} (the integrability of $A$ also follows from
Lemma~\ref{lem:int.action} and Prop.~\ref{prop:diff.action}(b)).
Combining our last two results we have the integration of
VB-algebroids:

\begin{theorem}\label{thm:integration}
Let $\Omega\then E$ be a VB-algebroid over $A\then M$, so that
$\Omega\then E$ is integrable. Then its source-simply-connected
integration $\Gamma\toto E$ carries a VB-groupoid structure over the
source-simply-connected Lie groupoid $G\toto M$ integrating $A\then
M$,
\begin{equation}\label{eq:VBalg2}
\begin{matrix}
\Gamma & \toto & E\\
\downarrow & & \downarrow\\
G& \toto & M,
\end{matrix}
\end{equation}
uniquely determined by the property that its differentiation is the
given VB-algebroid.
\end{theorem}

\begin{proof}
By Theorem~\ref{thm:VB-alg}, the given VB-algebroid is described by
an IM action on $\Omega\then E$, and since $\Gamma\toto E$ is its
source-simply-connected integration, it acquires a multiplicative
$\Rm$-action $h$ by Lemma~\ref{lem:int.action}. By
Prop.~\ref{prop:int.action}, we know that $h$ is regular, and
Theorem~\ref{thm:VB-gpd} concludes the proof.
\end{proof}


Though the integrability of the total algebroid in a VB-algebroid
implies that of the base algebroid, the
converse is not true. 
We illustrate this fact with an example.

\begin{example}\label{ex:nonint}
Let $M$ be a connected manifold, and let $\omega \in \Omega^2(M)$ be
a closed $2$-form such that its group of periods,
$$
\left \{ \int_\gamma \omega \, \Big | \ \gamma \in \pi_2(M) \right
\} \subset \mathbb{R},
$$
is not trivial. Let $E$ and $C$ denote the trivial line bundle $q:
\mathbb{R}_M=M\times \mathbb{R} \to M$, and consider the vector
bundle $\Omega = TM\oplus E \oplus C \to E$. For $X\in
\mathcal{X}(M)$ and $f: \mathbb{R}_M\to \mathbb{R}_M$ satisfying
$q\circ f=q$, let $\sigma_{X,f}$ be the section of $\Omega\to E$
given by $\sigma_{X,f}(e)=(X(q(e)),e,f(e))$. Then $\Omega$ carries a
unique Lie-algebroid structure such that its anchor map and bracket
satisfy
$$
\rho(X,e,c) = (X, 0) \; \in TE|_e\cong TM|_{q(e)}\times \mathbb{R},
\quad\text{ and}\quad
[\sigma_{X_1,c_1}, \sigma_{X_2,c_2}] = \sigma_{[X_1,X_2],f},
$$
where $c_i \in \mathbb{R}$ (viewed as constant maps $\mathbb{R}_M\to
\mathbb{R}_M$), and $f(e)=e \omega(X_1,X_2)|_{q(e)}$. The
$\Rm$-action on $\Omega$ defined by $h_\lambda(X,e,c)= (X, \lambda
e, \lambda c)$ defines a VB-algebroid structure on $\Omega\then E$
over $TM \then M$. Although the base Lie algebroid is clearly
integrable, $\Omega \then E$ is not. This follows from the
obstruction theory for integrability of \cite{CF}: one can check
that the monodromy group of $\Omega$ at $e\in E|_x=\mathbb{R}$
corresponds to the group of periods of $e \omega \in \Omega^2(M)$,
so any of its non-trivial elements accumulate at $0$ as $e$ goes to
$0$ (c.f. \cite[Thm.~4.1]{CF}).
\end{example}

This example is the starting point for the study of obstructions to
integrability of VB-algebroids, further developed in \cite{BrCO}.


As shown by the next proposition, the expected relations between
VB-algebroid and VB-groupoid maps via integration hold:

\begin{proposition}\label{prop:int.maps}
Let $\Gamma_1$ and $\Gamma_2$ be VB-groupoids.
\begin{enumerate}[(a)]
\item If $\Gamma_1$ is source-connected, then $\Phi:\Gamma_1\to\Gamma_2$
is a VB-groupoid map if and only if $\Phi'$ is a VB-algebroid map.
\item If $\Gamma_1$ is source-simply-connected, then there is a one-to-one
correspondence between VB-groupoid maps $\Gamma_1\to\Gamma_2$ and
VB-algebroid maps $A_{\Gamma_1}\to A_{\Gamma_2}$.
\end{enumerate}
\end{proposition}

\begin{proof}
Item $(a)$ follows from the characterization of VB-algebroid and
VB-groupoid maps by $\Rm$-equivariance, together with the uniqueness
of integration of Lie-algebroid maps when the domain is a
source-connected Lie groupoid, see Remark~\ref{rem:sconnect}. Part
$(b)$ is a direct consequence of Lie's second theorem.

\end{proof}

In general, given a source-simply-connected Lie groupoid $G$, Lie
subalgebroids of $A_G$ may not integrate to Lie subgroupoids of $G$,
see \cite{mm2}. We can obtain information about the integration of
VB-subalgebroids from the previous proposition:

\begin{corollary}\label{cor:VBsub}
Let $\Omega_1$ be a VB-subalgebroid of $\Omega_2$ defining, at the
level of basis algebroids, a Lie subalgebroid $(A_1\then
M_1)\hookrightarrow (A_2\then M_2)$. For $i=1,2$, let $\Gamma_i$ and
$G_i$ be source-simply-connected integrations of $\Omega_i$ and
$A_i$, respectively. Then $\Gamma_1$ is a VB-subgroupoid of
$\Gamma_2$ provided $G_1$ is a Lie subgroupoid of $G_2$.
\end{corollary}

\begin{proof}
Let $(\Phi,\phi): (\Gamma_1\toto E_1) \to (\Gamma_2\toto E_2)$ be
the groupoid map such that $\Phi': (\Omega_1\then E_1)
\to(\Omega_2\then E_2)$ is the subalgebroid inclusion. Since $\Phi'$
and $\dd \phi$ are injective on fibers, we see (from
Remark~\ref{rem:sscandmap}(b), applied to $\dd \Phi: T\Gamma_1\to
T\Gamma_2$) that $\Phi$ is an immersion. So it remains to check that
it is injective. By the previous proposition, $\Phi$ defines a
vector-bundle map $(\Gamma_1\to G_1) \to (\Gamma_2\to G_2)$, so it
can be identified with the restriction of its differential to the
vertical bundles. The immersion property then implies that $\Phi$ is
fibrewise injective. By assumption, $\Phi$ restricts to an injective
map $G_1\to G_2$, hence the result.
\end{proof}

This last corollary is used in the study of distributions \cite{jo}
and Dirac structures on Lie algebroids and groupoids \cite{ortiz}.

\begin{example}\label{dirac}
Let $A\then M$ be a Lie algebroid, and consider the VB-algebroid
$TA\oplus T^*A \then TM\oplus A^*$ over it. Let
 $L_A\hookrightarrow TA\oplus T^*A$ be a VB-subalgebroid over $A\then M$.
If $G$ is the source-simply-connected integration of $A$, then the
VB-groupoid $TG\oplus T^*G \toto TM\oplus A^*$ is the
source-simply-connected integration of $TA\oplus T^*A$. By
Corollary~\ref{cor:VBsub} the source-simply-connected integration of
$L_A$ defines a VB-subgroupoid $L_G\hookrightarrow TG\oplus T^*G$
over $G\toto M$. A special class of such VB-subalgebroids $L_A$ is
given by those which are, additionally, Dirac structures. In this
case, the VB-subgroupoids $L_G$ just defined are proven in
\cite[Thm.~5.2]{ortiz} to be Dirac structures as well.
\end{example}


We end this section with comments on the relation between
integration of VB-algebroids and representations up to homotopy.

\begin{remark}\label{int.ruth}
Representations up to homotopy of a Lie groupoid can be
differentiated to representations up to homotopy of its Lie
algebroid, see \cite{AS1}. The converse integration procedure is
considered in \cite{AS2}, but in a formal sense: representations up
to homotopy of a Lie algebroid are integrated to those of its
$\infty$-groupoid -- a global object associated to any Lie
algebroid. For an {\it integrable} Lie algebroid $A$, a natural
question is whether, or under which conditions, representations up
to homotopy integrate to those of its source-simply-connected Lie
groupoid $G$. Our result on integration of VB-algebroids provides
information about this problem: A representation of $A$ on $C \to E$
is integrable to one of $G$ if and only if the Lie algebroid
$\Omega= A\oplus E \oplus C\then E$ is integrable, where $\Omega$ is
the VB-algebroid corresponding to $C\to E$ (in the sense of the
results in \cite{GM,GM2} mentioned in Example~\ref{ex:reps} and
Remark~\ref{rmk:ruth}). For example, the adjoint representation of a
Lie algebroid is always integrable, and the representation up to
homotopy of $TM$ underlying Example~\ref{ex:nonint} is not
integrable (c.f. \cite[Prop~5.4]{AS2}). More on this topic can be
found in \cite{BrCO}.


\end{remark}


\section{Applications to double Lie algebroids}\label{sec:DLA}

In this section, following our previous results on VB-algebroids and
VB-groupoids, we discuss the Lie theory relating more general
objects, known as {\em double Lie algebroids} and {\em LA-groupoids}
\cite{mac-double1,mac-double2,mac-crelle}. Rather than treating
these objects directly, our approach is to focus on the dual
picture, in which we trade Lie algebroids for linear Poisson
structures. 
From this viewpoint, the objects to be considered are
VB-algebroids and VB-groupoids endowed with a compatible
Poisson structure, and our main goal is to extend our integration
result in Thm.~\ref{thm:integration} to this setting.

From an alternative perspective, following Theorems \ref{thm:VB-gpd}
and \ref{thm:VB-alg}, we will be considering regular actions on
objects known as {\em Poisson groupoids} and  {\em Lie bialgebroids}
\cite{Mac-Xu,We88}. We start the section by briefly discussing them.



\subsection{Interlude: Lie bialgebroids and Poisson groupoids}


A {\em Lie bialgebroid} is a pair of Lie-algebroid structures
$A\then M$ and $A^*\then M$ which are compatible in the sense that
$$
d_*[X,Y]=[d_*X,Y]+[X,d_*Y] \qquad \forall X, Y
\in\Gamma(\wedge^\bullet A),
$$
where $d_*$ is the differential in $\Gamma(\wedge^\bullet A)$
induced by the bracket of $A^*$, and $[\cdot,\cdot]$ denotes the
Schouten bracket induced by the bracket of $A$. One may verify that
the notion of Lie bialgebroid is symmetric in $A$ and $A^*$, see
e.g. \cite[Sec.~12.1]{Mac-book} for details.


By the duality between Lie-algebroid structures and linear Poisson
structures (see Section~\ref{subsec:linear}), a Lie bialgebroid is
the same as a Lie algebroid $A\then M$ equipped with a linear
Poisson structure $\pi$ on $A$ satisfying the following
compatibility condition {\cite{Mac-Xu}: the associated bundle map
$\pi^\#:(T^*A\to A) \to (TA\to A)$ is a Lie algebroid map with
respect to the tangent and cotangent Lie algebroids, $TA\then TM$
and $T^*A\then A^*$, see Example~\ref{ex:tangcot}. In other words,
$\pi$ defines a VB-algebroid map
\begin{equation}\label{eq:anchorbialg}
\begin{matrix}
T^*A & \then & A^*\\
\downarrow & & \downarrow\\
A& \then & M
\end{matrix}
\qquad\xto{\pi^\#}\qquad
\begin{matrix}
TA & \then & TM\\
\downarrow & & \downarrow\\
A& \then & M.
\end{matrix}
\end{equation}
We will denote Lie bialgebroids by pairs $(A\then M,\pi)$. A map of
Lie bialgebroids is a map of Lie algebroids which is also a Poisson
map.




If $A\then M$ is a Lie algebroid, then it becomes a bialgebroid with
the trivial Poisson structure. The following are less trivial
examples.

\begin{example}\label{ex:bialg}
Lie bialgebras are Lie bialgebroids over a point. Other examples are
associated with Poisson manifolds $(P,\pi)$: the tangent-lift
$\pi_T$ on $TP$, corresponding to the Lie algebroid structure on
$T^*P$, makes $(TP\then P,\pi_T)$ into a bialgebroid.

\end{example}


The global counterparts of Lie bialgebroids are Poisson groupoids
\cite{Mac-Xu2}. A {\it Poisson groupoid} \cite{Mac-Xu,We88} is a Lie
groupoid $G\toto M$ equipped with a Poisson structure $\pi$ which is
{\it multiplicative}, in the sense that the bundle map $\pi^\#: T^*G
\to TG$ is a Lie-groupoid map. In other words, $\pi$ gives rise to a
VB-groupoid map
\begin{equation}\label{eq:panchorgpd}
\begin{matrix}
T^*G & \toto & A^*\\
\downarrow & & \downarrow\\
G& \toto & M
\end{matrix}
\qquad\xto{\pi^\#}\qquad
\begin{matrix}
TG & \toto & TM\\
\downarrow & & \downarrow\\
G& \toto & M.
\end{matrix}
\end{equation}
A map of Poisson groupoids is a Lie-groupoid map that is also a Poisson map.

\begin{example}
Every Poisson-Lie group is a Poisson groupoid with a single object,
and every symplectic groupoid \cite{CDW} is a Poisson groupoid with
non-degenerate Poisson structure. These are two fundamental families
of examples.
\end{example}


If $(G\toto M,\pi_G)$ is a Poisson groupoid, then its Lie algebroid
$A_G\then M$ inherits a Poisson structure $\pi_A$ making it into a
Lie bialgebroid. One may obtain $\pi_A$ from $\pi_G$ as follows: by
applying the Lie functor to $\pi_G^\#$ and using the canonical
isomorphisms $j_G: TA_G\to A_{TG}$ and $\theta_G: A_{T^*G} \to
T^*A_G$ (see Remark~\ref{rmk:identifications}(c)), we define the
Poisson bivector $\pi_A$ on $A_G$ by
\begin{equation}\label{eq:piA}
(\pi_G^\#)'=j_G\circ \pi_A^\# \circ \theta_G.
\end{equation}


There is also an integration procedure going from Lie bialgebroids
to Poisson groupoids \cite{Mac-Xu2}: if $(A\then M,\pi_A)$ is a Lie
bialgebroid and $G\toto M$ is a source-simply connected Lie groupoid
integrating $A\then M$, then there exists a unique Poisson structure
$\pi_G$ on $G$ that makes it into a Poisson groupoid and induces
$\pi_A$ via \eqref{eq:piA}. This follows from Lie's second theorem:
$\pi_G$ is obtained by integrating the Lie algebroid map
$A_{T^*G}\to A_{TG}$ defined by the right-hand-side of
\eqref{eq:piA}.



We conclude this discussion with a direct proof of the integration
of Lie-bialgebroid maps, which completes the partial result in
\cite[Thm.~5.5.]{xu}:

\begin{proposition}\label{bialgmap}
Let $(G_i\toto M_i,\pi_{G_i})$ be Poisson groupoids, with Lie
bialgebroids $(A_i\then M_i, \pi_{A_i})$, $i=1,2$.
\begin{enumerate}[(a)]
\item Let $\Phi: G_1\to G_2$ be a Lie groupoid map. If it is a
map of Poisson groupoids, then $\Phi'$ is a map of Lie bialgebroids,
and the converse holds if $G_1$ is source-connected;
\item
When $G_1$ is source-simply-connected, any Lie-bialgebroid map
$A_1\to A_2$ integrates to a unique map of Poisson groupoids $G_1\to
G_2$.
\end{enumerate}
\end{proposition}
\begin{proof}
Notice that for $(a)$, it is enough to assume that $G_1$ is
source-connected and show that $\Phi$ is a Poisson map with respect
to $\pi_{G_1}$ and $\pi_{G_2}$ if and only if $\Phi'$ is a Poisson
map relative to $\pi_{A_1}$ and $\pi_{A_2}$.

From \eqref{eq:piA}, we know that
\begin{equation}\label{eq:piA2}
(\pi_{G_2}^\#)' = j_{G_2}\circ \pi_{A_2}^\# \circ \theta_{G_2}.
\end{equation}

By functoriality of pullbacks (see Remarks~\ref{rem:pbgrp} and
\ref{rem:pbalg}), $\pi_{G_2}^\#: T^*G_2\to TG_2$ gives rise to a
VB-groupoid map $\Phi^*T^*G_2 \to \Phi^* TG_2$, that we keep
denoting by $\pi_{G_2}^\#$. With this simplified notation, we write
$(\pi_{G_2}^\#)': (\Phi')^* A_{T^*G_2} \to (\Phi')^* A_{TG_2}$, see
Remark~\ref{rmk:identifications}(a). Similarly, one can apply the
pullback functor $(\Phi')^*$ to all maps on the right-hand side of
\eqref{eq:piA2}, and view \eqref{eq:piA2} as an equality of
VB-algebroid maps $(\Phi')^* A_{T^*G_2} \to (\Phi')^* A_{TG_2}$.

Consider the tangent VB-groupoid map $\dd\Phi: TG_1\to \Phi^*TG_2$,
and its dual $(\dd\Phi)^*: \Phi^*T^*G_2 \to T^*G_1$. Differentiating
the composition $\dd\Phi \circ \pi_{G_1}^\# \circ (\dd\Phi)^*:
\Phi^*T^*G_2\to \Phi^*TG_2$ leads to a VB-algebroid map
\begin{align}\label{eq:pmap}
(\dd\Phi \circ \pi_{G_1}^\# \circ (\dd\Phi)^*)' & =(\dd\Phi)'\circ
j_{G_1}\circ \pi_{A_1}^\# \circ \theta_{G_1} \circ ((\dd\Phi)^*)' \\
& = j_{G_2}\circ \dd(\Phi') \circ \pi_{A_1}^\# \circ (\dd\Phi')^*
\circ \theta_{G_2}, \nonumber
\end{align}
where we have used \eqref{eq:piA} and Remark~\ref{rmk:identifications}(c).

Note that $\Phi$ is a Poisson map if and only if
$$
\pi_{G_2}^\# = \dd\Phi \circ \pi_{G_1}^\# \circ (\dd\Phi)^*,
$$
as an equality of maps $\Phi^*T^*G_2 \to \Phi^* TG_2$. Both maps are
Lie-groupoid maps and $\Phi^* T^*G_2$ is source-connected (since it
is a VB-groupoid over $G_1$, and $G_1$ is assumed to be
source-connected, see Remark~\ref{rem:sscandmap}(a)). By the
uniqueness result in Remark~\ref{rem:sconnect}, the last equation
holds if and only if
$$
(\pi_{G_2}^\#)' = (\dd\Phi \circ \pi_{G_1}^\# \circ (\dd\Phi)^*)'.
$$
Comparing with \eqref{eq:piA2} and \eqref{eq:pmap}, we see that this is equivalent to
$$
\pi_{A_2}^\# = \dd(\Phi') \circ \pi_{A_1}^\# \circ (\dd \Phi')^*,
$$
as an equality of maps $(\Phi')^*T^*A_2 \to (\Phi')^* TA_2$, which
is the condition for $\Phi'$ being a Poisson map.

Finally, part $(b)$ is an immediate consequence of $(a)$.
\end{proof}

An alternative approach to this last result can be found in
\cite[Sec.~1.5]{luca2}.


\subsection{LA-groupoids and double Lie algebroids: the dual viewpoint}\label{DLA}

We now consider certain generalizations of VB-algebroids and
VB-groupoids, in which the vector-bundle structures are enhanced to
be Lie algebroids.


An {\em LA-groupoid} \cite{mac-double1} consists of a VB-groupoid
$\Gamma \toto E$ over $G\toto M$, along with Lie algebroid
structures $\Gamma\then G$ and $E\then M$, satisfying compatibility
conditions saying that the groupoid structure maps are Lie-algebroid
morphisms (an alternative definition will be given below). We depict
an LA-groupoid as
\begin{equation}\label{eq:LAgrp}
\begin{matrix}
\Gamma & \toto & E\\
\Downarrow & & \Downarrow\\
G& \toto & M.
\end{matrix}
\end{equation}

The duality between Lie algebroids and linear Poisson structures
provides an alternative viewpoint to LA-groupoids in terms of their
duals, that we now recall.


A {\em PVB-groupoid} \cite{mk2} consists of a VB-groupoid
$\Gamma\toto E$ over $G\toto M$ and a Poisson structure $\pi$ on
$\Gamma$ which is multiplicative (i.e., $(\Gamma,\pi)$ is a Poisson
groupoid) and linear with respect to $\Gamma\to G$. PVB-groupoids
will be written as
$$
\left(\begin{matrix}
\Gamma & \toto & E\\
\downarrow & & \downarrow\\
G& \toto & M
\end{matrix},\pi\right)
$$

As proven in \cite[Thm.~3.14]{mk2}, the compatibility conditions
relating the groupoid and algebroid structures on an LA-groupoid
\eqref{eq:LAgrp} are equivalent to saying that the dual VB-groupoid
$\Gamma^*\toto C^*$ over $G\toto M$ is a PVB-groupoid with respect
to the linear Poisson structure dual to $\Gamma\then G$. So, through
VB-groupoid duality, one obtains a one-to-one correspondence between
LA-groupoids and PVB-groupoids.


\begin{example}\label{ex:PVBg}
Examples of PVB-groupoids include, e.g., the cotangent groupoid of
any Lie groupoid (equipped with the Poisson structure of its
canonical symplectic form) as well as the tangent groupoids to
Poisson groupoids, equipped with the tangent-lift Poisson structure
(see Remark~\ref{rem:poisson}).
\end{example}


One advantage of passing from LA-groupoids to PVB-groupoids is that
the latter admit a simple characterization in terms of
$\Rm$-actions, resulting from Prop.~\ref{linear} and
Theorem~\ref{thm:VB-gpd}:

\begin{proposition}\label{prop:pvb-gr}
A PVB-groupoid is the same as a Poisson groupoid $(\Gamma\toto
E,\pi)$ equipped with a regular action $h: \Rm\action (\Gamma\toto
E)$ such that $h_\lambda:(\Gamma\toto E,\pi)\to (\Gamma\toto
E,\lambda\pi)$ is a map of Poisson groupoids for all $\lambda\neq
0$.
\end{proposition}


We now consider the analogous infinitesimal objects. A {\em double
Lie algebroid} \cite{mac-crelle} consists of a VB-algebroid
$\Omega\then E$ over $A\then M$ equipped with additional Lie
algebroid structures $\Omega\then A$ and $E\then M$, depicted
\begin{equation}\label{eq:DLA}
\begin{matrix}
\Omega & \then & E\\
\Downarrow & & \Downarrow\\
A& \then & M,
\end{matrix}
\end{equation}
satisfying the following conditions:
\begin{enumerate}[(i)]
 \item $\Omega\then A$ is a VB-algebroid over $E\then M$,
 \item the Lie-algebroid structure on the vertical dual $\Omega^*_A\then C^*$
(see \eqref{eq:dualalg}), together with the linear Poisson structure
$\pi_A$ induced by $\Omega\then A$, define a Lie bialgebroid.
\end{enumerate}
Note that (ii) can be equivalently stated in terms of $\Omega_E^*\to
C^*$, interchanging the roles of vertical and horizontal
VB-algebroids in \eqref{eq:DLA}.


Once again, it will be profitable to make use of the duality between
Lie algebroids and linear Poisson structures and pass to the dual
picture.

A {\em PVB-algebroid} consists of a VB-algebroid and a Poisson
structure $\pi$ on the total space $\Omega$ which is linear with
respect to $\Omega \to A$ and such that $(\Omega\then E,\pi)$ is a
Lie bialgebroid. We will use the notation
$$\left(\begin{matrix}
\Omega & \then & E\\
\downarrow & & \downarrow\\
A& \then & M
\end{matrix},\pi\right).
$$


One can directly check that vertical duality of VB-algebroids
establishes a one-to-one correspondence between double Lie
algebroids and PVB-algebroids,
\begin{equation}\label{eq:1-1pvb}
\begin{matrix}
\Omega & \then & E\\
\Downarrow & & \Downarrow\\
A & \then & M
\end{matrix}
\qquad \stackrel{1:1}{\longleftrightarrow} \qquad
\left(\begin{matrix}
\Omega^*_A & \then & C^*\\
\downarrow& & \downarrow\\
A & \then & M
\end{matrix},\pi_A\right),
\end{equation}
analogously to what happens for LA-groupoids and PVB-groupoids
\cite[Thm.~3.14]{mk2}.

\begin{remark}
From the duality properties of VB-algebroids and the fact that Lie
bialgebroids are self-dual, one sees that both the horizontal and
vertical duals of a double Lie algebroid are PVB-algebroids. Thus,
while the vertical dual of a PVB-algebroid is a double Lie
algebroid, its horizontal dual is again a PVB-algebroid.
\end{remark}


\begin{example}
Analogously to Example~\ref{ex:PVBg}, one can see that the cotangent
Lie algebroid of any Lie algebroid is naturally a PVB-algebroid
(with respect to the canonical symplectic structure). The tangent
Lie algebroid to any Lie bialgebroid $(A,\pi)$ is a PVB-algebroid,
with Poisson structure given by the tangent lift of $\pi$.
\end{example}


Just as PVB-groupoids, PVB-algebroids admit a simple description in
terms of regular actions, following Prop.~\ref{linear} and
Theorem~\ref{thm:VB-alg}:

\begin{proposition}\label{thmpvbalg}
PVB-algebroids are equivalently described as Lie bialgebroids
$(\Omega \then E,\pi)$ endowed with a regular IM-actions $h:\Rm
\action (\Omega \then E)$ such that $h_\lambda:(\Omega \then
E,\pi)\to(\Omega \then E,\lambda\pi)$ is a Poisson map $\forall
\lambda\neq 0$.
\end{proposition}

The characterizations of PVB-algebroids and PVB-groupoids in
Props.~\ref{prop:pvb-gr} and \ref{thmpvbalg} will be useful in
describing their relation via differentiation and integration.


\begin{remark}
PVB-groupoids and PVB-algebroids can be characterized by means of
their Poisson-anchor maps $\pi^\#$. Indeed, tangent and cotangent
bundles of VB-groupoids and VB-algebroids inherit {\em triple
structures}, which can be encoded in cubical diagrams as the
cotangent cube \eqref{cot.cub} (see \cite[Fig.~5]{mk2}). Combining
\eqref{eq:lps}, \eqref{eq:anchorbialg} and \eqref{eq:panchorgpd},
one can see that a Poisson structure defines a PVB-groupoid or a
PVB-algebroid if and only if the map $\pi^\#$ preserves the
structure of the underlying cubes. This characterization for
PVB-algebroids is essentially \cite[Thm.~3.9]{mac-crelle}, and as
explained there, it implies that $\pi^\#$ automatically yields an
algebroid map $(\Omega_A^*\then C^*)\to(TA\then TM)$, simplifying
some redundacy in the original definition of double Lie algebroids,
see e.g. \cite[Sec.~2]{dla}.
\end{remark}

\subsection{Lie theory}

We finally explain how double Lie algebroids are related to
LA-groupoids under differentiation and integration. We will do so by
first studying the dual picture, i.e., the Lie theory relating
PVB-algebroids and PVB-groupoids.



\begin{proposition}\label{prop:PVBlie}
Consider a VB-groupoid $\Gamma\toto E$ over $G\toto M$, and let
$\pi$ be a multiplicative Poisson structure on $\Gamma$. Consider
the corresponding VB-algebroid $A_\Gamma\then E$ over $A_G\then M$
and Lie bialgebroid $(A_\Gamma,\pi_{A_\Gamma})$. If $\pi$ is linear
with respect to $\Gamma\to G$, then $\pi_{A_\Gamma}$ is linear with
respect to $A_\Gamma\to A_G$, and the converse holds provided
$\Gamma\toto E$ is source-connected.
\end{proposition}

\begin{proof}
Note that if $(A\then M,\pi)$ is a Lie bialgebroid, then so is
$(A\then M,\lambda\pi)$ for any $\lambda\in \mathbb{R}$, since
$\lambda \pi$ is a Poisson structure and we can write
$(\lambda\pi)^\#$ as the composition
$$\begin{matrix}
T^*A & \then & A^*\\
\downarrow & & \downarrow\\
A& \then & M
\end{matrix}
\qquad\xto{\pi^\#}\qquad
\begin{matrix}
TA & \then & TM\\
\downarrow & & \downarrow\\
A& \then & M
\end{matrix}
\qquad\xto{k_\lambda}\qquad
\begin{matrix}
TA & \then & TM\\
\downarrow & & \downarrow\\
A& \then & M,
\end{matrix}$$
where $k$ denotes the regular action associated with the vector
bundle $TA\to A$. The analogous result holds for Poisson groupoids.
Moreover, one may directly check that if $(G\toto M,\pi_G)$
integrates the bialgebroid $(A\then M,\pi_A)$, then $(G\toto M,
\lambda\pi_G)$ integrates $(A\then M,\lambda\pi_A)$, for $\lambda\in
\mathbb{R}$.

Let $h$ be the regular multiplicative action on $\Gamma\toto E$
defining its VB-groupoid structure, so that $h'$ defines the
VB-algebroid structure on $A_\Gamma\then E$ (see
Cor.~\ref{cor:diff.VB}). The linearity of $\pi$ with respect to
$\Gamma\to G$ is equivalent to $h_\lambda:(\Gamma,\pi)\to
(\Gamma,\lambda\pi)$ being a map of Poisson groupoids for all
$\lambda \neq 0$ (see Prop.~\ref{prop:pvb-gr}), while the linearity
of $\pi_{A_\Gamma}$ with respect to $A_\Gamma\to A_G$ is equivalent
to $h_\lambda': (A_\Gamma, \pi_{A_\Gamma})\to (A_\Gamma, \lambda
\pi_{A_\Gamma})$ being a map of Lie bialgebroids (see
Prop.~\ref{thmpvbalg}). The result now follows from the integration
of bialgebroid maps in Prop.~\ref{bialgmap}.
\end{proof}

The previous proposition immediately gives rise to a Lie functor
from PVB-groupoids to PVB-algebroids and implies the following
integration result:

\begin{corollary}\label{cor:integPVB}
If the total algebroid of a PVB-algebroid is integrable, then its
source-simply-connected integration inherits a unique  PVB-groupoid
structure whose differentiation is the given PVB-algebroid.
\end{corollary}


Let us now consider LA-groupoids and double Lie algebroids. By
applying the Lie functor to the horizontal groupoid structures of an
LA-groupoid \eqref{eq:LAgrp}, one obtains a VB-algebroid
$A_\Gamma\then E$ over $A_G\then M$. A key observation is that there
is also a natural Lie-algebroid structure $A_\Gamma\then A_G$,
described in \cite[Thm.~2.14]{mac-double2}, so one can consider the
diagram of Lie algebroids
\begin{equation}\label{eq:DLA2}
\begin{matrix}
A_\Gamma & \then & E\\
\Downarrow & & \Downarrow\\
A_G& \then & M.
\end{matrix}
\end{equation}
This can be shown to be a double Lie algebroid \cite{mac-double2},
yielding a Lie functor from LA-groupoids to double Lie algebroids;
we will revisit this fact below and complement it with the
corresponding integration result.


\begin{remark}\label{rem:sublie}
The Lie algebroid $A_\Gamma\then A_G$ referred to above can be also
directly described as the Lie-algebroid fibred product (see
Prop.~\ref{prop:fp.alg}),
\begin{equation}\label{eq:fibredLA}
\xymatrix{
(A_\Gamma\then A_G) \ar[r] \ar[d] &  (T\Gamma\then TG) \ar[d]^{(T\sour_\Gamma,T\sour_G) \times (q^\Gamma,q^G)}\\
 (E\then M) \ar[r]_{(0_E,0_M)\times (u_\Gamma,u_G)} & (TE\then TM)\times (\Gamma\then G).}
\end{equation}
Indeed, the Lie algebroid resulting from this fibred product is
uniquely characterized by the fact that it sits in $T\Gamma\then TG$
as a Lie subalgebroid; since the one in
\cite[Thm.~2.14]{mac-double2} also satisfies this property
\cite[Prop.~5.5]{ortiz}, they must coincide.
\end{remark}


The following diagram illustrates our strategy to describe the Lie
theory relating double Lie algebroids and LA-groupoids:
$$
\xymatrix@C=40pt{
\text{LA-groupoids} \ar@{-->}[d]_{Lie} \ar@{<->}[r]^{duality}&\text{PVB-groupoids} \ar[d]^{Lie}\\
\text{Double Lie algebroids}
\ar@{<->}[r]_{duality}&\text{PVB-algebroids} }
$$
In order to follow the dotted arrow backwards, we will pass to the
dual framework and use the integration result in
Corollary~\ref{cor:integPVB}. But we first need to verify  that the
Lie functors on each side correspond to one another under duality.

\begin{proposition}\label{lie&duality}
The square above commutes up to a canonical natural isomorphism.
\end{proposition}
\begin{proof}
Starting with an LA-groupoid \eqref{eq:LAgrp}, let us consider its dual PVB-groupoid
$$
\left(\begin{matrix}
\Gamma^* & \toto & C^*\\
\downarrow& & \downarrow\\
G & \toto & M
\end{matrix},\pi \right).
$$
By Prop.~ \ref{prop:PVBlie}, after applying the Lie functor we get a PVB-algebroid
$$
\left(\begin{matrix}
A_{\Gamma^*} & \then & C^*\\
\downarrow& & \downarrow\\
A_G & \then & M
\end{matrix},\pi_{A_{\Gamma^*}} \right).
$$
The pairing $\Gamma\times_G\Gamma^*\to \mathbb{R}$ leads to a
pairing $A_\Gamma\times_{A_G} A_{\Gamma^*}\to \mathbb{R}$ (via the
Lie functor, see Remark~\ref{rmk:identifications}(b)), and hence an
identification of VB-algebroids over $A_G\then M$,
\begin{equation}\label{eq:phi}
\phi: A_{\Gamma^*}\to (A_\Gamma)^*_{A_G}.
\end{equation}
This induces a PVB-algebroid structure on the VB-algebroid
$(A_\Gamma)^*_{A_G}\then C^*$ over $A_G\then M$, and hence, by
duality \eqref{eq:1-1pvb}, a double Lie algebroid
$$
\begin{matrix}
A_{\Gamma} & \then & E\\
\Downarrow& & \Downarrow\\
A_G & \then & M.
\end{matrix}
$$
It remains to check that the Lie algebroid $A_\Gamma\then A_G$
agrees with the one defined by \eqref{eq:fibredLA}. Equivalently, we
should verify that \eqref{eq:phi} is a Poisson isomorphism
$$
\phi: (A_{\Gamma^*},\pi_{A_{\Gamma^*}})\to ((A_\Gamma)^*_{A_G}, \bar{\pi}),
$$
where  $\bar{\pi}$ denotes the Poisson structure dual to the Lie
algebroid defined in \eqref{eq:fibredLA}.

To show that, recall that there is also a pairing (defined by the
tangent functor) $T\Gamma\times_{TG}T\Gamma^*\to \mathbb{R}$, which
leads to an identification $\Phi: T\Gamma^*\to (T\Gamma)^*_{TG}$.
Denoting by $\iota_{A_\Gamma}: A_\Gamma\to T\Gamma$ and
$\iota_{A_{\Gamma^*}}: A_{\Gamma^*}\to T\Gamma^*$ the natural
inclusions, one may directly verify that the maps $\Phi$ and $\phi$
are related by
$$
\phi= (\iota_{A_\Gamma})^*\circ \Phi\circ \iota_{A_{\Gamma^*}},
$$
where we view $(\iota_{A_\Gamma})^*$ as the dual relation to
$\iota_{A_\Gamma}$ and consider the composition of relations on the
right-hand side. Note that the relation $(\iota_{A_\Gamma})^*$ is
one-to-one, and defined over the whole image of the map $\Phi\circ
\iota_{A_{\Gamma^*}}$, so their composition is a map.

Endowing $T\Gamma^*$ with the tangent lift of $\pi$ (cf. Remark
~\ref{rem:poisson}) and $(T\Gamma)^*_{TG}$ with the Poisson
structure dual to the tangent Lie algebroid $T\Gamma\then TG$, it
follows from \cite[Thm.~10.3.14]{Mac-book} that $\Phi$ is a Poisson
isomorphism, so its  graph is coisotropic. The inclusion
$\iota_{A_{\Gamma^*}}$ is also a Poisson map with respect to
$\pi_{A_{\Gamma^*}}$ and the tangent lift of $\pi$, see e.g.
\cite[Sec.~6.3]{bc} (cf. \cite[Prop.~10.3.12 \&
Thm.~12.3.8]{Mac-book}). Since the dual relation
$(\iota_{A_\Gamma})^*$ is also coisotropic, see
Remark~\ref{rem:sublie},  and the composition of coisotropic
relations is coisotropic \cite{We88}, the graph of $\phi$ is
coisotropic, so it is a Poisson map.
\end{proof}

We conclude with the integration result.


\begin{theorem}\label{thm:LADLA}\
Consider a double Lie algebroid \eqref{DLA} for which the horizontal
Lie algebroid $\Omega\then E$ is integrable. Then its
source-simply-connected integration $\Gamma\toto E$ fits into an
LA-groupoid
\begin{equation}\label{eq:DLA3}
\begin{matrix}
\Gamma & \toto & E\\
\Downarrow & & \Downarrow\\
G& \toto & M,
\end{matrix}
\end{equation}
where $G\toto M$ is the source-simply-connected integration of
$A\then M$, uniquely determined by the property that its
differentiation is the given double Lie algebroid.
\end{theorem}

\begin{proof}
The dual VB-algebroid $\Omega^*_A\then C^*$ is integrated by the
dual VB-groupoid $\Gamma^*\toto C^*$ (see Prop.~\ref{lie&duality}),
and it is source-simply-connected (see
Remark~\ref{rem:sscandmap}(a)). This VB-groupoid inherits a
PVB-groupoid structure by Prop.~\ref{prop:PVBlie}. By dualizing it,
we obtain an LA-groupoid \eqref{eq:LAgrp} corresponding to
\eqref{eq:DLA3}.
\end{proof}

It would be interesting to use this theorem to extend the discussion
in Remark~\ref{rmk:ruth} to the context of representations up to
homotopy encoded by double Lie algebroids, as studied in
\cite{GJMM}.

A natural further step is the integration of LA-groupoids to double
Lie groupoids \cite{mac-double1}, as considered in \cite{luca}. We
hope to address this issue in a separate work.

\appendix

\section{Fibred products of Lie groupoids and Lie algebroids}

\def\A12{A_1 \times A_2}
\def\transverse{\pitchfork}
\def\eps{\epsilon}
\def\toto{\rightrightarrows}
\def\xto{\xrightarrow}
\def\xfrom{\xleftarrow}
\def\then{\Rightarrow}
\def\R{\mathbb R}
\def\into{\hookrightarrow}
\def\Lie{{\rm Lie}}



We present in this appendix a criterion for the existence of fibred
products in the categories of Lie algebroids and Lie groupoids,
extending and organizing some previous results in the literature. We
also study the behavior of fibred products under the Lie functor.


Our criterion is based on the following notion. Two smooth maps
$f_i:M_i\to M$, $i=1,2$, form a {\em good pair} if their
set-theoretic fibred product $M_{12}:=M_1 \times_{M} M_2 \subset M_1
\times M_2$ is an embedded submanifold with the expected tangent
space, i.e.,  for all $(x_1,x_2)\in M_{12}$ with
$x=f_1(x_1)=f_2(x_2)$ the following sequence is exact:
\begin{equation}\label{eq:exact}
0 \longrightarrow T_{(x_1,x_2)}M_{12} \longrightarrow T_{x_1}M_1
\times T_{x_2}M_2 \stackrel{\dd f_1-\dd f_2}{\longrightarrow} T_xM.
\end{equation}
We refer to the resulting manifold $M_{12}$ as a {\em good fibred
product}, for it satisfies the universal property and behaves well
with respect to the topologies and the tangent spaces (see e.g.
\cite{survey}). The paradigmatic example of a good pair is given by
transverse maps. Another example is given by embedded submanifolds
with clean intersection.

In this appendix, by a {\em submanifold} we mean an injective
immersion. Submanifold are usually identified with a subset of a
manifold $M$, equipped with a smooth structure, for which the
inclusion map is an injective immersion.

\subsection{The groupoid case}\label{subsec:subgrp}


A {\em Lie subgroupoid} of $G\toto M$  is a Lie groupoid
$\tilde{G}\toto \tilde{M}$ along with a Lie-groupoid map
$(\tilde{G}\toto \tilde{M})\to(G\toto M)$ which is an injective
immersion on objects and on arrows.


When studying fibred products, it is convenient to have an
alternative viewpoint. Given a Lie groupoid $G\toto M$, let
$(\tilde{G}\toto \tilde{M}) \subseteq (G\toto M)$ be a set-theoretic
subgroupoid, defined by restrictions of the structure maps of $G$.
Assume that the following conditions hold:
\begin{enumerate}
\item $\tilde{G}\subseteq G$ and $\tilde{M}\subseteq M$ are
submanifolds;
\item The restriction of the source map to $\tilde{G}$, ${\sour}: \tilde{G}\to \tilde{M}$, is a
submersion;
\item The structure maps of $\tilde{G}\toto \tilde{M}$ are smooth.
\end{enumerate}

It is not hard to see that Lie subgroupoids are equivalent to
set-theoretic subgroupoids satisfying these three conditions. We
remark that, in many situations, (3) automatically follows from (1)
and (2), e.g. when $\tilde{G}$ and $\tilde{M}$ are embedded. We also
observe that there are set-theoretic subgroupoids satisfying (1) and
(3), but which fail to be Lie subgroupoids by not satisfying (2)
(though this cannot happen when the subgroupoid is
source-connected). We will give an example below.

\begin{remark}
In order to consider the smoothness of the multiplication map, it is
implicitly required in (3) that $\tilde{G}^{(2)}$ sits in $G^{(2)}$
as a submanifold. Condition (2) guarantees this fact, but the next
example shows that it is not necessary.
\end{remark}

\begin{example}
Let $G\toto M$ be the Lie groupoid induced by the submersion
$\pi_1:\R^2\to\R$ (projection on the first factor): in this case
$M=\R^2$ and an arrow in $G$ consists of a pair of points in $\R^2$
on the same vertical line. Define $\tilde{M}=C_1\cup C_2\subset M$
as the union of the two curves
$$
C_1=\{(t^3,t): -1< t <1\} \qquad\text{and}\qquad C_2=\{(t,2): -1 < t
< 1\},
$$
and define $\tilde{G}\subset G$ as the space of arrows whose source
and target lie in $\tilde{M}$. Then $\tilde{G}\toto \tilde{M}$ is a
set-theoretic subgroupoid, and $\tilde{M}\subset M$ and
$\tilde{G}\subset G$ are embedded submanifolds. One may also
directly verify that $\tilde{G}^{(2)} \subset G^{(2)}$ is  an
embedded submanifold, so (1) and (3) above are satisfied. However,
the differential of $\sour:\tilde{G}\to \tilde{M}$ is not surjective
at the point $((0,0),(0,2))\in \tilde{G}$, so condition (2) above
does not hold.
\end{example}



The next lemma uses the following fact: Given a connected manifold
$M$ and a smooth map $f:M\to M$ such that $f^2=f$, its image
$f(M)\subset M$ is an embedded submanifold and $T_x
 f(M)=\dd_xf(T_xM)$ for all $x\in M$, see \cite[Thm.~1.13]{nodg} for
details.

\begin{lemma}
\label{lemma:good.pairs.groupoids} Let $(F_i,f_i):(G_i\toto
M_i)\to(G\toto M)$, $i=1,2$, be two Lie-groupoid maps. If $F_1,F_2$
is a good pair, then so is $f_1,f_2$.
\end{lemma}

\begin{proof}
Let us denote the set-theoretic fibred-product of $F_1$, $F_2$
(resp. $f_1$, $f_2)$ by $G_{12}$ (resp. $M_{12}$), and consider the
maps $\sour:= (\sour_1,\sour_2): G_{12}\to M_{12}$ and
$u:=(u_1,u_2): M_{12}\to G_{12}$, where $\sour_i, u_i$ are the
source and unit maps of $G_i$, $i=1,2$. Then $u\sour : G_{12}\to
G_{12}$ satisfies $(u\sour)^2=u\sour$, hence its image $u(M_{12})$
is an embedded submanifold of $G_{12}$, hence of $G_1\times G_2$,
and then of $u(M_1\times M_2)$.

Regarding the condition on the tangent spaces, we have to show that the sequence
$$
0 \to T_{(x_1,x_2)}M_{12} \to T_{x_1}M_1 \times T_{x_2}M_2 \to T_xM
$$
is exact (cf. \eqref{eq:exact}). But it follows from $\sour u=\id$
that this last sequence is a direct summand of
$$
0 \to T_{(x_1,x_2)}G_{12} \to T_{x_1}G_1 \times T_{x_2}G_2 \to T_xG,
$$
which is exact by hypothesis, and hence the result.
\end{proof}

We are now ready to consider fibred products of Lie groupoids.

\begin{proposition}
\label{prop:fp.gpd} Given a good pair of Lie groupoid maps as in
Lemma~\ref{lemma:good.pairs.groupoids}, the fibred-product manifolds
$G_{12}$ and $M_{12}$ define an embedded Lie subgroupoid of the
product groupoid,
$$
(G_{12}\toto M_{12})\subset(G_1\times G_2\toto M_1\times M_2).
$$
Moreover, this Lie groupoid satisfies the universal property of the
fibred product in the category of Lie groupoids.
\end{proposition}

\begin{proof}
Since $G_{12}\subset G_1\times G_2$ and $M_{12}\subset M_1\times
M_2$ are embedded submanifolds, it remains to show that source map
$\sour$ of $G_1\times G_2$ restricts to a submersion $\tilde{\sour}:
G_{12}\to M_{12}$.

Given $g = (g_1,g_2) \in G_{12}$ with source $x=(x_1,x_2)\in
M_{12}$, denote by $K_{g}$ and $K'_{g}$ the kernels of the maps
$$
\dd\sour:T_{ g}(G_1\times G_2)\to T_{ x}(M_1\times M_2) \quad \text{
and } \quad \dd\tilde{\sour} :T_{ g}G_{12}\to T_{ x}M_{12},
$$
respectively. Note that $K'_{ g}= K_{ g}\cap T_{ g}G_{12}$. Since we
know that $\tilde{\sour} : G_{12}\to M_{12}$ is a submersion close
to the units (as a consequence of $\sour u = \id$), it is enough to
show that $\dim K'_{ g}\leq \dim K'_{u( x)}$. We will show that
$\dd(R_{ g^{-1}}) (K'_{ g})\subset K'_{ u(x)}$ (here $R_g$ denotes
right-translation), and since $\dd(R_{ g^{-1}})$ is injective, the
result follows.

We have
$$K'_{g}= K_{g}\cap T_{g}G_{12}\subset T_{g}(G_1\times G_2)= T_{g_1}G_1\times T_{g_2}G_2$$
and
$$T_{g}G_{12}= T_{g_1}G_1\times_{T_{F_1(g_1)} G} T_{g_2}G_2.$$
Given $v=(v_1,v_2)\in K_{g}$, and given $h=(h_1,h_2)\in G_{12}$
composable with $g$, we have
$\dd(R_h)(v)=(\dd(R_{h_1})v_1,\dd(R_{h_2})v_2)$. If $v\in K'_{g}$,
then $\dd F_1(v_1)=\dd F_2(v_2)$, and
$$
\dd F_1(\dd(R_{h_1})(v_1)) = \dd(R_{F_1(h_1)}(\dd
F_1(v_1))=\dd(R_{F_2(h_2)}(\dd F_2(v_2)) = \dd
F_2(\dd(R_{h_2})(v_2)),
$$
from where $\dd(R_h)(v)\in K'_{gh}$, and hence
$\dd(R_h)(K'_{g})\subset K'_{gh}$ as desired.
\end{proof}

\begin{remark}
Special cases of the last proposition have appeared in the
literature. For instance, in the case of transverse maps (a
particular type of good pair), the existence of fibred products of
Lie groupoids is stated 
in \cite[pp.~123]{mm}. Under even stronger
assumptions, such fibred products were proven to exist e.g. in
\cite[Prop.~2.1]{luca} (under an additional ``source transversality
condition'') and \cite[Prop.~2.4.14]{Mac-book} (for groupoid maps
which are ``fibrations'').
A number of  good pairs which are not transverse maps appear in this
paper, see e.g. Remark~\ref{rmk:kergrp}.
\end{remark}


\subsection{The algebroid case}


By a {\em vector subbundle} of a vector bundle $A\to M$ we mean a
vector bundle $\tilde{A}\to \tilde{M}$ and injective immersions
$\tilde{A}\into A$, $\tilde{M}\into M$ defining a vector-bundle map.
We say that $\tilde{A}\to \tilde{M}$ is {\em embedded} if
$\tilde{M}\subset M$ is so, and consequently also $\tilde{A}\subset
A$. A {\em Lie subalgebroid} of $A\then M$ is a vector subbundle
equipped with a Lie-algebroid structure for which the inclusion is a
Lie-algebroid map (c.f. \cite[Def.~1.2]{HiMa}).

In order to deal with fibred products, it is convenient to have a
criterion to identify vector subbundles which inherit the structure
of a Lie subalgebroid. Given an embedded vector subbundle
$\tilde{A}\to \tilde{M}$ of a Lie algebroid $A\then M$, let us
consider the following compatibility conditions with the anchor and
bracket:
\begin{enumerate}
 \item[$(i)$] $\rho(\tilde{A})\subset T\tilde{M}$, and
 \item[$(ii)$] if $X,Y\in\Gamma(A)$ are such that $X|_{\tilde{M}},Y|_{\tilde{M}}\in\Gamma(\tilde{A})$,
 then $[X,Y]|_{\tilde{M}}\in\Gamma(\tilde{A})$.
\end{enumerate}

\begin{remark} These compatibility conditions also make sense for non-embedded subbundles:
in this case, the conditions $\rho(\tilde{A})\subset T\tilde{M}$
and $X|_{\tilde{M}}\in\Gamma(\tilde{A})$ implicitly require that the
induced set-theoretic functions $\tilde{A}\to T\tilde{M}$ and
$\tilde{M}\to \tilde{A}$ are smooth, which is automatic in the
embedded case.
\end{remark}


These compatibility conditions imply the following  additional
property:

\begin{lemma}\label{lem:iii}
\label{lemma:iii} Given a Lie algebroid $A\then M$ and a subbundle
$\tilde{A}\to \tilde{M}$ satisfying property (i) above, the
following holds:
\begin{enumerate}
\item[(iii)] If $X,Y\in\Gamma(A)$ satisfy $X|_{\tilde{M}}=0$ and $Y|_{\tilde{M}}\in\Gamma(\tilde{A})$,
then $[X,Y]|_{\tilde{M}}=0$.
\end{enumerate}
\end{lemma}
\begin{proof}
We can work locally and assume that $A\to M$ is trivial, with a
basis of sections $\{e_1,\dots,e_r\}$. Let
$$
[e_i,e_j] = \sum_{k}c^{ij}_k e_k, \qquad  c^{ij}_k \in C^\infty(M).
$$
Let $X,Y\in\Gamma(A)$ be such that $X|_{\tilde{M}}=0$ and
$Y|_{\tilde{M}}\in\Gamma(\tilde{A})$. We have to show that
$[X,Y](x)=0$ for all $x\in \tilde{M}$. We can write $X = \sum_i a_i
e_i$ and $Y = \sum_j b_j e_j$, with $a_i,b_j\in C^\infty(M)$.
Then their bracket is
$$
[X,Y]= \sum_k \left( \sum_{i,j} a_i b_j c^{i,j}_k + \rho(X) b_k -
\rho(Y) a_k \right) e_k.
$$
Given $x\in \tilde{M}$, $X(x)=0$ and, equivalently, $a_i(x)=0$ for
all $i$. The result now follows from $\rho(Y)$ being tangent to
$\tilde{M}$.
\end{proof}

It directly follows that a subbundle $\tilde{A}\to \tilde{M}$
satisfying $(i)$ and $(ii)$ above naturally inherits a Lie-algebroid
structure from $A\then M$, where the bracket is defined by locally
extending sections of $\tilde{A}$ to sections of $A$, using the
bracket on $A$, and then restricting to $\tilde{M}$; this operation
is well defined by Lemma~\ref{lemma:iii}. The structure on
$\tilde{A}\then \tilde{M}$ clearly makes the inclusion into a
Lie-algebroid map. Conversely, any Lie subalgebroid satisfies $(i)$
and $(ii)$, and its Lie-algebroid structure agrees with the one
induced from these properties. This notion of Lie subalgebroid
appears in \cite[Def.~4.3.14]{Mac-book} (where condition $(iii)$ is
required as an extra axiom).

A simple, but relevant, property of Lie subalgebroids, used
recurrently, is that a Lie-algebroid map $(B\then N) \to (A\then M)$
whose image lies in a Lie subalgebroid $\tilde{A}\then \tilde{M}$
gives rise to a Lie-algebroid map $(B\then N) \to (\tilde{A}\then
\tilde{M})$.

In order to study fibred products of Lie algebroids, we first
discuss vector bundles.

\begin{lemma}
\label{lemma:good.pairs.alg} Let $(F_i,f_i):(E_i\to M_i)\to(E\to
M)$, $i=1,2$, be two vector-bundle maps. The smooth maps $F_1,F_2$
form a good pair if and only if $f_1,f_2$ form a good pair and the
vector-bundle map
\begin{equation}\label{eq:E12map}
(F_1)\pi_1 - (F_2)\pi_2:E_1\times E_2|_{M_{12}}\to E
\end{equation}
has constant rank. (Here $\pi_i: E_1\times E_2\to E_i$ is the
projection, $i=1,2$.)
\end{lemma}

\begin{proof}
Assuming that $(F_1,F_2)$ is good, the same arguments used in
Lemma~\ref{lemma:good.pairs.groupoids} show that $(f_1,f_2)$ is also
good, hence $M_{12}\subset M$ is embedded with the expected tangent
space. Since the kernel of the map \eqref{eq:E12map} is the manifold
$E_{12} = E_1\times_M E_2$, it must have constant rank.

Conversely, if $(f_1,f_2)$ is a good pair and the rank of the map
\eqref{eq:E12map} is constant, then $M_{12}\subset M_1\times M_2$ is
embedded with the expect tangent space, $E_{12}$ (which is the
kernel of the map) is also an embedded submanifold, and a vector
subbundle. It only remains to show that it has the expected tangent
space.


For any vector bundle $q: E\to M$, we can identify a fiber
$E_{q(v)}$ with $\ker(\dd q)_v \subset T_vE$, and in this way
 we obtain a short exact sequence of complexes,
$$\xymatrix@R=10pt{
0 \ar[r] \ar[d] & (E_{12})_{(x_1,x_2)} \ar[r] \ar[d] & (E_1)_{x_1} \times (E_2)_{x_2} \ar[r] \ar[d] & E_x \ar[d]\\
0 \ar[r] \ar[d] & T_{(v_1,v_2)}E_{12}  \ar[r] \ar[d] & T_{v_1}E_1 \times T_{v_2}E_2 \ar[r] \ar[d]& T_vE \ar[d]\\
0 \ar[r] & T_{(x_1,x_2)}M_{12} \ar[r]& T_{x_1}M_1 \times T_{x_2}M_2
\ar[r]& T_xM. }
$$
The top sequence is exact by assumption, the bottom one is exact
because $f_1,f_2$ is a good pair, so the middle one is also exact,
proving the result.
\end{proof}

When the conditions in the  previous lemma hold, the vector bundle
$E_{12}\to M_{12}$ satisfies the universal property of the fibred
product in the category of vector bundles.

We now move to Lie algebroids.

\begin{proposition}
\label{prop:fp.alg} Let $(F_i,f_i):(A_i\then M_i)\to(A\then M)$,
$i=1,2$, be two Lie-algebroid maps so that the pair $(F_1,F_2)$ is
good. Then the vector-bundle fibred product
$$
(A_{12}\to M_{12})\subset(A_1\times A_2\then M_1\times M_2)
$$
is an embedded Lie subalgebroid of the product Lie algebroid.
Moreover, the Lie algebroid $(A_{12}\then M_{12})$ satisfies the
universal property of the fibred product in the category of Lie
algebroids.
\end{proposition}

\begin{proof}
We have to show that $A_{12}\to M_{12}$ satisfies conditions $(i)$
and $(ii)$.

Regarding $(i)$, since $(F_1,f_1)$ and $(F_2,f_2)$ are Lie-algebroid
maps, we have
$$
(\dd f_1) \rho_1=\rho F_1 \quad\text{ and }\quad (\dd
f_2)\rho_2=\rho F_2,
$$
which implies that $(\rho_1,\rho_2)(A_1\times_A A_2)\subset
TM_1\times_{TM}TM_2=T M_{12}$, the last equality following from
$f_1,f_2$ being a good pair.

Regarding $(ii)$,  given $X,Y\in\Gamma(A_1\times A_2)$ such that
$X|_{M_{12}},Y|_{M_{12}}\in\Gamma(A_{12})$, we have to show that the
same holds for $[X,Y]$. Consider the product map
$$
(F,f)=(F_1\times F_2,f_1\times f_2):(A_1\times A_2\then M_1\times
M_2) \to (A\times A\then M\times M).
$$
We can always write $FX=\sum_i a_i (X_i f)$ for some functions
$a_i\in C^\infty(M_1\times M_2)$ and sections $X_i\in\Gamma(A\times
A)$, as an identity between sections of the pullback bundle.
Consider the diagonal subbundle $\Delta_A\to \Delta_M$ of $A\times
A\to M\times M$. Since $X|_{M_{12}}\in\Gamma(A_{12})$, we see that
$FX(M_{12})\subset \Delta_A$; also, we can choose the sections $X_i$
such that $X_i(\Delta_M)\subset\Delta_A$. We can proceed analogously
for $Y$.

Since $F$ is an algebroid map, we have the equation
$$
F[X,Y] = \sum_{i,j} a_ib_j([X_i,Y_j]f) + (\rho(X) b_j)(Y_jf) -
(\rho(Y)a_i)(X_if),
$$
and by using that $\Delta_A\to \Delta_M$ is a Lie subalgebroid of
$A\times A\to M\times M$, we conclude that
$F[X,Y](M_{12})\subset\Delta_A$, and hence the result.
\end{proof}

A brief discussion on fibred products of Lie algebroids can be found
in \cite[pp.207]{HiMa}, and further details (in the case of
transverse maps) in \cite[Prop.~2.3]{luca}.



\subsection{Fibred products and the Lie functor}


When passing from Lie groupoids to Lie algebroids via the Lie
functor, recall the notations $A_G=\Lie(G)$ and $F'=\Lie(F)$ for
maps.
For a Lie groupoid $G\toto M$, there is a natural splitting
\begin{equation}\label{eq:split}
TG |_{u(M)}=A_G \oplus TM,
\end{equation}
so that, for a groupoid map $(F,f):G\to \tilde{G}$, we obtain a
decomposition
\begin{equation}\label{eq:splitmap}
\dd F = (F',\dd f): (A_G)_x \oplus TM_x \to (A_{\tilde{G}})_{f(x)}
\oplus T_{f(x)}\tilde{M}.
\end{equation}

\begin{proposition}\label{prop:diff.fp}
Let $(F_i,f_i):(G_i\toto M_i)\to(G\toto M)$, $i=1,2$, be
Lie-groupoid maps such that $F_1$ and $F_2$ form a good pair. Then
the induced pair $F_i':A_{G_i}\to A_G$, $i=1,2$, is also good, and
the canonical map (arising from the universal property of fibred
products) $ A_{G_1\times_G G_2}\to A_{G_1} \times_{A_G} A_{G_2},$ is
an isomorphism, which, upon the obvious identification $T(G_1\times
G_2)=TG_1\times TG_2$, is just the identity.
\end{proposition}

\begin{proof}

Denote by $G_{12}$ the fibred-product Lie groupoid, which exists by
Prop.~\ref{prop:fp.gpd}, and let $A_{12}=\Lie(G_{12})$. According to
Lemma~\ref{lemma:good.pairs.alg}, to see that $F_1'$, $F_2'$ is a
good pair we need to show that $f_1$, $f_2$ form a good pair and
that the vector-bundle map
\begin{equation}\label{eq:Fp}
F_1' \pi_1 - F_2'\pi_2:A_{G_1}\times A_{G_2}|_{M_{12}}\to A
\end{equation}
has constant rank. The first condition follows from $F_1$, $F_2$
forming a good pair, see Lemma~\ref{lemma:good.pairs.groupoids}.
Given $(x_1,x_2)\in M_{12}$, since $F_1$, $F_2$ form a good pair,
and using  the splitting \eqref{eq:split}, we have the following
exact sequence (cf. \eqref{eq:exact}):
$$
0 \to (A_{12} \oplus T M_{12})_{(x_1,x_2)} \to (A_{G_1}\times
A_{G_2})_{(x_1,x_2)}\oplus(T M_1 \times T M_2)_{(x_1,x_2)} \to
(A_G\oplus TM)_x.
$$
Using \eqref{eq:splitmap}, we see that the kernel of $(F_1'\pi_1 -
F_2'\pi_2)_{(x_1,x_2)}$ is exactly $(A_{12})_{(x_1,x_2)}$, so the
map \eqref{eq:Fp} has constant rank.

Knowing that $(F_1',F_2')$ is a good pair, we conclude that the
fibred-product Lie algebroid is well-defined (by
Prop.~\ref{prop:fp.alg}), and by construction it agrees with
$A_{12}$.
\end{proof}


\end{document}